\documentclass[twoside,a4paper,11pt,leqno]{article}
\usepackage{latexsym,amsfonts,amssymb,amsmath,amscd}
= -0mm \evensidemargin = 10pt
\usepackage[T1]{fontenc}
\usepackage[latin1]{inputenc}
\usepackage{amsmath}
\usepackage{epsfig}
\usepackage{wasysym}
\usepackage{amsfonts}
\usepackage{multicol}
\usepackage{graphicx}
\usepackage{wasysym}

\def \R {I\!\!R}

\def\B{{\cal B}}
\def\N{\hbox{\bf N}}

\def\p{\partial}

\def\k{\kappa}

\newenvironment{proof}
{\medskip\noindent{\bf Proof.\/}} {\null \hfill $\Box$
\par\medskip}

\def \deml{{\bf Proof of Lemma}}

\def \diam{\hbox{\em diam }}

\newcommand{\ba}{\begin{eqnarray}}
\newcommand{\ea}{\end{eqnarray}}

\newcommand{\basn}{\begin{eqnarray*}}
\newcommand{\easn}{\end{eqnarray*}}

\newtheorem{thoa}{\bf Theorem A}
\newtheorem{theorem}{\bf Theorem}[section]

\newtheorem{lemma}{\bf Lemma}[section]
\newtheorem{proposition}[theorem]{\bf Proposition}

\newtheorem{remark}{\bf Remark}[section]

\numberwithin{equation}{section}
\title{Problem with critical Sobolev exponent and with weight}
\author{R.Hadiji\quad and \quad
H.Yazidi \\ UFR des Sciences et Technologie, CNRS UMR 8050,\\
Université Paris 12 - Val-de-Marne,\\
61, avenue du Général de Gaulle,\\
94010 Créteil Cedex.\\(hadiji@univ-paris12.fr,
yazidi@univ-paris12.fr) }
\date{}
\begin{document}
\maketitle
\begin{abstract}
We consider the problem: $-\textsl{div}(p\nabla u)={u}^{q
-{1}}+\lambda{u}$, $u > 0$ in $\Omega$, $u=0$ on $\partial \Omega$.
Where $\Omega$ is a bounded domain in $\R^{n}$, ${n}\geq{3}$, $ p:
\bar{\Omega}\longrightarrow \R$ is a given positive weight such that
$p\in H^{1}(\Omega)\cap C(\bar{\Omega})$, $\lambda$ is a real
constant and $q=\frac{2n}{n-2}$. We study the effect of the behavior
of $p$ near its minima and the impact of the geometry of domain on
the existence of solutions for the above problem.
\end{abstract}
\section{Introduction}
In this paper we study the following problem:\ba \left\{
\begin{array}{lll} -\textsl{div}(p(x)\nabla u)={u}^{q
-{1}}+\lambda{u} &\textrm{in $\Omega$,}\\
\hspace{25.5mm}u > 0 &\textrm{in $\Omega$,}\\
\hspace{25mm}u=0  &\textrm{on $\partial \Omega$,}
\end{array}
\right. \label{eqhy1} \ea where $\Omega$ is a bounded domain in
$\R^{n}$, ${n}\geq{3}$, $ p: \bar{\Omega}\longrightarrow \R$ is a
given positive weight such that $p\in H^{1}(\Omega)\cap
C(\bar{\Omega})$, $\lambda$ is a real constant and
$q=\frac{2n}{n-2}$ is the critical exponent for the Sobolev
embedding of $H_{0}^{1}(\Omega)$ into $L^{q}(\Omega)$.\\In
\cite{bn}, Brezis and Nirenberg treated the case where $p$ is
constant. They proved, in particular, the existence of a solution of
(\ref{eqhy1}) for $0<\lambda<\lambda_{1}$ if $n\ge4$ and for
$\lambda^{*}<\lambda<\lambda_{1}$ if $n=3$, where $\lambda_{1}$ is
the first eigenvalue of $-\Delta$ on $\Omega$ with zero Dirichlet
boundary condition and $\lambda^{*}$ is a positive constant.\\In
this paper, we extend this result to the general case of where $p$ is not constant. The study of problem (\ref{eqhy1}), shows that the existence of solutions depends, apart from parameter $\lambda$, on the behavior of p near its minima and on the geometry of the domain $\Omega$.\\
Set $p_{_{0}}=\min\{p(x),\,x\in\bar{\Omega}\}$, we suppose that
$p^{-1}(\{p_{_{0}}\})\cap\Omega\not=\emptyset$ and let $a\in
p^{-1}(\{p_{_{0}}\})\cap\Omega$.\\In the first part of this work, we
study the effect of the behavior of $p$ near its minima on the
existence of solution for our problem. The method that is mostly
relied upon, apart from the identities of Pohozeav, is the
adaptations to the new context of the arguments developed in
\cite{bn}.\\We assume that, in a neighborhood of $a$, $p$ behaves
like
\begin{equation}
p(x)=p_{_{0}}+\beta_{k}\vert
x-a\vert^{k}+|x-a|^{k}\theta(x),\label{eqhy0}\end{equation}
with $k>0$, $\beta_{k}>0$ and $\theta(x)$ tends to $0$ when $x$ tends to $a$.\\
%\,\textrm{with $k>0$,\,$\beta_{k}> 0$ and $\theta(x)$\, tend to
%$0$\, when $x$\, tend to $a$ .}\label{eqhy0}
%\end{equation}
Note that the parameter $k$ will play an essential role in the study
of our problem. Indeed, 2 appears as a critical value for $k$. More
precisely the case $k> 2$ is treated by a classical procedure,
however the case $0<k\leq 2$ is less easily accessible. Therefore,
in this case, we restrict ourself to the case where $p$ satisfies
the additional condition
\begin{equation}
 k\beta_{k}\leq\frac{\nabla p(x).(x-a)}{|x-a|^{k}}\quad \textrm{a.e
 \,$x\in\Omega$.}
 \label{eqhy00}
\end{equation}
Let us notice that if $p$ is sufficiently smooth, then condition
(\ref{eqhy0}) follows directly from Taylor's expansion of $p$ near
$a$.\\The fact that $2$ is a critical value for $k$ appears clearly
in dimension $n=4$, therefore, in this dimension and with the aim of
obtaining more explicit results, we assume moreover that $\theta$
satisfies $\int_{B(a,1)}\frac{\theta(x)}{|x-a|^{4}}dx<\infty$. Let
us emphasize that this last condition is not necessary to prove the
existence of solutions.\\Moreover, in dimension $n=3$, the problem
is more delicate, then we treat it in a particular case; more
precisely for $p(x)=p_{_{0}}+\beta_{k}|x-a|^{k}$, $k>0$.\\The first
result of this paper is the following
\begin{theorem}~\\
Assume that $p\in H^{1}(\Omega)\cap C(\bar{\Omega})$ satisfies
(\ref{eqhy0}). Let $\lambda_{1}^{div}$ be the first eigenvalue of
$-\textsl{div}(p(x)\nabla .)$ on $\Omega$ with zero Dirichlet
boundary condition, we have\\1)If $n\geq 4$ and $k>2$, then for
every $\lambda \in
]0,\lambda_{1}^{\textsl{div}}[$ there exists a solution of~(\ref{eqhy1}).\\
2)If $n\geq 4$ and $k=2$, then there exists a constant
$\tilde{\gamma}(n)=\frac{(n-2)n(n+2)}{4 (n-1)}\beta_{2}$ such that
for every $\lambda\in ]\tilde{\gamma}(n), \lambda_{1}^{div}[$ there
exists a
solution of~(\ref{eqhy1}).\\
3)If $n=3$ and $k\ge 2$, then there exists a constant $\gamma(k)>0$ such that for every $\lambda\in]\gamma(k),\lambda_{1}^{\textsl{div}}[$ there exists a solution of~(\ref{eqhy1}).\\
4)If $n\geq 3$, $0<k<2$ and $p$ satisfies the condition
(\ref{eqhy00}) then there exists $\lambda^{*}\in
[\tilde{\beta_{k}}\frac{n^{2}}{4}, \lambda_{1}^{\textsl{div}}[$,
where $\tilde{\beta_{k}}=\beta_{k}\min[(\diam \Omega)^{k-2},1]$,
such that for any $\lambda \in
]\lambda^{*},\lambda_{1}^{\textsl{div}}[$ problem (\ref{eqhy1})
admits a solution.\\
5)If $n \ge 3$ and $k>0$, then for every $\lambda\leq 0$ there is no
minimizing solution of equation (\ref{eqhy1}).\\
6)If $n\ge3$ and $k>0$, then there is no solution of problem
(\ref{eqhy1}) for every $\lambda\ge\lambda_{1}^{div}.$\label{thhy2}
\end{theorem}\vspace{-1.3mm}
\begin{remark}~\\
In general, the intervals $]\tilde{\gamma}(n), \lambda_{1}^{div}[$
in 2) and $[\tilde{\beta_{k}}\frac{n^{2}}{4},
\lambda_{1}^{\textsl{div}}[$ in 4), may be empty. But there are some
sufficient conditions for which the above intervals are nonempty:\\
1) If
$p_{_{0}}>\displaystyle\frac{n(n-4)}{(n-1)(n-2)^{2}}\beta_{2}\left(\diam\Omega\right)^{2}$,
then $\tilde{\gamma}(n)<\lambda_{1}^{div}$.\\Notice that this
condition is always true if $n$ is rather large.\\
2) If
$p_{_{0}}>\displaystyle\frac{\tilde{\beta_{k}}n^{2}}{(n-2)^{2}\left(\diam\Omega\right)^{2}}$,
then $\tilde{\beta_{k}}\frac{n^{2}}{4}<\lambda_{1}^{div}$.\\
\label{remarque}
\end{remark}

The second part of this work is dedicated to the study of the effect
of the geometry of the domain on the existence of solutions of our
problem. More precisely, since for $\lambda=0$ and $p\in
H^{1}(\Omega)\cap C(\bar{\Omega})$ satisfying $\nabla p(x).(x-a)>0$
a.e in $\Omega$, the problem (\ref{eqhy1}) does not have a solution
for a starshaped domain about $a$, we will modify the geometry of
$\Omega$ in order to find a solution. Therefore, let
$\Omega\subset\R^{n},\, n\ge3$ be a starshaped domain about $a$ and
let $\varepsilon>0$, we will study the existence of solution of the
problem\basn {\bf({I}_{\varepsilon})}\quad\left\{
\begin{array}{lll} -\textsl{div}(p(x)\nabla u)={u}^{q
-{1}} &\textrm{in $\Omega_{\varepsilon}$,}\\
\hspace{25.5mm}u > 0 &\textrm{in $\Omega_{\varepsilon}$,}\\
\hspace{25mm}u=0  &\textrm{on $\partial \Omega_{\varepsilon}$,}
\end{array}
\right.\easn where $\Omega_{\varepsilon}=\Omega\setminus
\bar{B}(a,\varepsilon)$.\\For $p\equiv 1$ and $\lambda=0$, the
problem $(\ref{eqhy1})$ has been first investigated in \cite{c} and
an interesting result of existence has been proved for domains with
holes. In \cite{bac}, this last result is extended to all domains
having "nontrivial" topology (in a suitable sense). This
nontrivially condition (which covers a large class of domains) is
only sufficient for the solvability but not necessary as shown by
some examples of contractible domains $\Omega$ for which
(\ref{eqhy1}) has solutions (see \cite{d}, \cite{di},
\cite{pa}).\\In other direction, \cite{le} shows that the solution
of \cite{c}, on a domain with a hole of diameter $\varepsilon$ and
center $x_{0}$, concentrates at the point $x_{0}$. In \cite{h}, the
author generalized the result of \cite{c} for the case where $u^{q}$
is replaced by $u^{q}+\mu u^{\alpha}$, where $\mu\in\R$ and
$1<\alpha<q$.\\In this work, we consider the case where $p\in
H^{1}(\Omega)\cap C(\bar{\Omega})$ and satisfying $\nabla
p(x).(x-a)>0$ a.e on $\Omega\setminus\{a\}$. The method we use in
this part is an adaptation of those used in \cite{c} and \cite{h}.
More particularly, we use the min-max techniques and a variant of
the Ambrosetti-Rabinowitz theorem, see \cite{ar}.\\The second result
of this paper is the following
\begin{theorem}~\\
\label{thhy0}There exists
$\varepsilon_{0}=\varepsilon_{0}(\Omega,p)>0$ such that for
$0<\varepsilon<\varepsilon_{0}$ the problem $(I_{\varepsilon})$ has
at least one solution in $H^{1}_{0}(\Omega_{\varepsilon})$.
\end{theorem}
The rest of this paper is divided into three sections. In Section 2
some preliminary results will be established. Section 3 and Section
4 are devoted respectively to the proof of Theorem~\ref{thhy2} and
the proof of Theorem~\ref{thhy0}.
\section{Some preliminary results}
We start by recalling some notations which will be frequently used
throughout the rest of this paper. First, we define
$$S=\inf_{u\in H_{0}^{1}(\Omega), \|u\|_{q}=1}\|\nabla u\|_{2}^{2}$$that corresponds
to the best constant for the Sobolev embedding
$H_{0}^{1}(\Omega)\subset L^{q}(\Omega)$. Let us denote by
$U_{{a},\varepsilon}$ an extremal function for the Sobolev
inequality
\begin{equation*}
U_{{a},\varepsilon}(x)=\frac{1}{(\varepsilon +
\vert{x-a}\vert^{2})^\frac{n-2}{2}} , \ x\in\R^{n}. \label{eqhy-1}
\end{equation*}
We set
\begin{equation}
u_{{a},\varepsilon}(x)=\zeta(x)U_{{a},\varepsilon}(x)\,,\, x\in
\R^{n}, \label{eqhy-1}
\end{equation}where $\zeta \in C^{\infty}_{0}(\bar{\Omega})$ is a fixed
function such that ${0}\leq{\zeta}\leq{1}$, and $\zeta\equiv{1}$ in
some
neighborhood of\, $a$\, included  in $\Omega$.\\
We know from \cite{bn} that \ba \|\nabla
u_{a,\varepsilon}\|_{2}^{2}=\frac{K_{1}}{\varepsilon^{\frac{n-2}{2}}}+O(1),
\label{eq2} \ea \ba
\parallel
u_{a,\varepsilon}\parallel^{2}_{q}=\frac{K_{2}}{\varepsilon^{\frac{n-2}{2}}}+O(\varepsilon)
\label{eqhy3} \ea and \ba {\parallel
u_{a,\varepsilon}\parallel^{2}_{2}}= \left\{
\begin{array}{l} \frac{K_{3}}{\varepsilon^{\frac{n-4}{2}}}+O(1) \
if\
n\geq{5}\\[\medskipamount]
\frac{\omega_{_{4}}}{2} |\log {\varepsilon }| +O(1)\ if\ n={4}
\end{array}
\right.\label{eqhy4} \ea where $K_{1}$ and $K_{2}$ are positive
constants with $\frac{K_{1}}{K_{2}}=S$, $\omega_{_{4}}$ is the area
of $S^{3}$ and
$K_{3}=\displaystyle\int_{\R^{n}}\frac{1}{(1+|x|^{2})^{n-2}}dx$.\\~\\
We shall state some auxiliary results.\\For $p\in
C^{1}(\bar{\Omega})$ or $p\in H^{1}(\Omega)\cap C(\bar{\Omega})$ and
$\nabla p(x).(x-a)\ge 0$ a.e $x\in\Omega$, we consider
$$
\alpha(p)=\frac{1}{2}\inf_{u\in
H_{0}^{1}(\Omega),u\not=0}\frac{\int_{\Omega}\nabla p(x).(x-a)\vert
\nabla u\vert^{2}dx}{\int_{\Omega}\vert u\vert^{2}dx}.
$$We easily see that $\alpha(p)\in[-\infty,+\infty[$, and we have the following result
\begin{proposition}~\\
{\bf 1)}\,If $p\in C^{1}(\Omega)$ and if there exists $b\in\Omega$
such
that $\nabla p(b)(b-a)<0$, then $\alpha(p)=-\infty$.\\
{\bf 2)}\,If $p\in H^{1}(\Omega)\cap C(\bar{\Omega})$ satisfying
(\ref{eqhy0}) and $ \nabla
p(x).(x-a)\ge0$ \,a.e $x\in\Omega$, we have\\
{\bf 2.a)}\,If $k>2$ and $p\in C^{1}(\Omega)$, then $\alpha(p)=0$ for all  $n\geq 3$.\\
{\bf 2.b)}\,If $0<k\leq 2$ and $p$ satisfies condition
(\ref{eqhy00})
 then for all $n\geq 3$ we have
$$\frac{k}{2}\beta_{k}\left(\frac{n+k-2}{2}\right)^{2}\left(\diam\Omega\right)^{k-2}
\leq \alpha(p).$$ \label{pr1}
\end{proposition}
\begin{proof} We start by proving 1). Set $q(x)=\nabla p(x).(x-a)$, $\forall
x\in\Omega$ and let $\varphi\in C_{0}^{\infty}(\R^{n})$ such that
$0\leq \varphi\leq 1$ on $\R^{n}$, $\varphi\equiv 1$ on the ball
$\{x, |x|<r\}$, and $\varphi\equiv 0$ outside the ball $\{x,
|x|<2r\}$, where $r<1$ is a positive constant .\\
Set $\varphi_{j}(x)=\varphi(j(x-b))$ for $j\in\N^{*}$. We have \basn
\alpha(p)&\leq&\frac{1}{2}\frac{\int_{\Omega}q(x)|\nabla
\varphi_{j}(x)|^{2}dx}{\int_{\Omega}| \varphi_{j}|^{2}dx}\\[\medskipamount]
&\leq&\frac{1}{2}\frac{\int_{B(b,\frac{2r}{j})}q(x)|\nabla
\varphi_{j}(x)|^{2}dx}{\int_{B(b,\frac{2r}{j})}|
\varphi_{j}|^{2}dx}.\easn Using the change of variable $y=j(x-b)$,
we get \basn
\alpha(p)&\leq&\frac{j^{2}}{2}\frac{\int_{B(0,2r)}q(\frac{y}{j}+b)|\nabla
\varphi(x)|^{2}dx}{\int_{B(0,2r)}| \varphi|^{2}dx}.\easn Applying
the Dominated Convergence Theorem, we obtain \basn
\alpha(p)&\leq&\frac{j^{2}}{2}\left[q(b)\frac{\int_{B(0,2r)}|\nabla
\varphi(x)|^{2}dx}{\int_{B(0,2r)}| \varphi|^{2}dx}+o(1)\right].\easn
Letting $j\rightarrow\infty$, we deduce the desired result.\\Now we
will prove 2.a).\\Using (\ref{eqhy0}) and since $p\in C^{1}(\Omega)$
in a neighborhood $V$ of $a$, we write \ba
p(x)=p_{_{0}}+\beta_{k}\vert
x-a\vert^{k}+\theta_{1}(x),\label{equat9}\ea where $\theta_{1}\in
C^{1}(V)$ is such that \ba\lim_{x\rightarrow
a}\frac{\theta_{1}(x)}{|x-a|^{k}}=0 \label{equat10}.\ea Looking at
(\ref{equat10}), we deduce that there exists $0<r<1$, such that
\begin{equation}\theta_{1}(x)\leq |x-a|^{k}\quad \forall x\in
B(a,2r).\label{equat11}\end{equation} Let $\varphi\in
C_{0}^{\infty}(\R^{n})$ be a function such that $0\leq \varphi\leq
1$ on $\R^{n}$, $\varphi\equiv 1$ on the ball $\{x, |x|<r\}$, and
$\varphi\equiv 0$ outside the ball $\{x, |x|<2r\}$. Set
$\varphi_{j}(x)=\varphi(j(x-a))$ for $j\in\N^{*}$, we have \basn
0\leq\alpha(p)\leq \frac{1}{2}\frac{\int_{\Omega}\nabla
p(x).(x-a)|\nabla \varphi_{j}(x)|^{2}dx}{\int_{\Omega}|
\varphi_{j}|^{2}dx}.\easn Using (\ref{equat9}), we see that \basn
0\leq\alpha(p)\leq\frac{k\beta_{k}}{2}\frac{\int_{_{B(a,\frac{2r}{j})}}|x-a|^{k}|\nabla
\varphi_{j}(x)|^{2}dx}{\int_{_{B(a,\frac{2r}{j})}}|
\varphi_{j}|^{2}dx}+\frac{1}{2}\frac{\int_{_{B(a,\frac{2r}{j})}}\nabla\theta_{1}(x).(x-a)|\nabla
\varphi_{j}(x)|^{2}dx}{\int_{_{B(a,\frac{2r}{j})}}|
\varphi_{j}|^{2}dx}.\easn Performing the change of variable
$y=j(x-a)$, and integrating by parts the second term of the right
hand side, we obtain \basn
0\leq\alpha(p)\leq\frac{k\beta_{k}}{2j^{k-2}}\frac{\int_{_{B(0,2r)}}|y|^{k}|\nabla
\varphi(y)|^{2}dx}{\int_{_{B(0,2r)}}|
\varphi|^{2}dx}+\frac{j}{2}\frac{\int_{_{B(0,2r)}}\theta_{1}(\frac{y}{j}+a)\nabla(y
|\nabla \varphi(y)|^{2})dx}{\int_{_{B(0,2r)}}| \varphi|^{2}dx}.
\easn Using (\ref{equat11}), we write \basn
0\leq\alpha(p)\leq\frac{k\beta_{k}}{2j^{k-2}}\frac{\int_{_{B(0,2r)}}|y|^{k}|\nabla
\varphi(y)|^{2}dx}{\int_{_{B(0,2r)}}|
\varphi|^{2}dx}+\frac{1}{2j^{k-1}}\frac{\int_{_{B(0,2r)}}|y|^{k}\nabla(
|\nabla \varphi(y)|^{2}y)dx}{\int_{_{B(0,2r)}}| \varphi|^{2}dx}.
\easn Therefore, for $k>2$ we deduce that $\alpha(p)=0$, and this
finishes the proof of this case.\\
Now, in order to prove 2.b), we need to recall the following Hardy's
inequality, see for example \cite{ckn} or Theorem 330 in \cite{hlp}.
\begin{lemma}~\\
\label{lmhy1} Let $t\in\R$ such that $t+n>0$, we have $\forall u\in
H_{0}^{1}(\Omega)$
$$
\int_{\Omega}|x|^{t}|u|^{2}dx\leq (\frac{2}{n+t})^{2}\int_{\Omega}|
x.\nabla u|^{2}| x|^{t}dx.
$$
Moreover the constant $(\frac{2}{n+t})^{2}$ is optimal and is not
achieved.
\end{lemma}
\medskip
%%%%%%%%%
\vspace{10mm} Now we prove 2.b). Since $p$ satisfies (\ref{eqhy00}),
we have for all $u\in H_{0}^{1}(\Omega)\setminus\{0\}$,
\begin{equation}
\begin{split}
\frac{\int_{\Omega}\nabla p(x).(x-a)|\nabla
u(x)|^{2}dx}{\int_{\Omega}|u(x)|^{2}dx}\ge
k\beta_{k}\frac{\int_{\Omega} |x-a|^{k}|\nabla
u(x)|^{2}dx}{\int_{\Omega}|u(x)|^{2}dx}.
\end{split}
\nonumber
\end{equation}
By applying the last Lemma for $0<k=2+ t\leq 2$, we find
\begin{equation}
\begin{split}
\frac{\int_{\Omega}\nabla p(x).(x-a)|\nabla
u(x)|^{2}dx}{\int_{\Omega}|u|^{2}dx}\geq k\beta_{k}
\left(\frac{n+k-2}{2}\right)^{2}\left(\diam \Omega\right)^{k-2}.
\end{split}
\nonumber
\end{equation}
This implies that $\alpha (p)\geq \frac{k}{2}\beta_{k}
(\frac{n+k-2}{2})^{2}(\diam \Omega)^{k-2}$.
\end{proof}Let us give the following non-existence result
\begin{proposition}~\\
We assume that $\alpha(p)>-\infty$. There is no solution for
(\ref{eqhy1}) when $\lambda\leq \alpha(p)$ and $\Omega$ is a
starshaped domain about $a$. \label{poho}
\end{proposition}
\begin{proof}
This follows from Pohozev's identity. Suppose that $u$ is a solution
of (\ref{eqhy1}). We first multiply (\ref{eqhy1}) by $\nabla
u(x).(x-a)$, next we integrate over $\Omega$ and we obtain
\begin{equation}
\int_{\Omega}u^{q-1}\nabla u(x).(x-a)
dx=-\frac{n-2}{2}\int_{\Omega}|u(x)|^{q}dx, \label{porog1}
\end{equation}
\begin{equation}
\lambda\int_{\Omega}u \nabla u(x)
.(x-a)dx=-\frac{n}{2}\lambda\int_{\Omega}|u(x)|^{2}dx \label{porog2}
\end{equation}
and
\begin{equation}\begin{array}{lll} \displaystyle\int_{\Omega}-\textsl{div}(p(x)\nabla u)\nabla
u(x).(x-a)dx&=&\displaystyle\hspace{-2mm}-\frac{n-2}{2}\int_{\Omega}\hspace{-0.8mm}p(x)
|\nabla
u(x)|^{2}dx\\[\bigskipamount]&-&\displaystyle\frac{1}{2}\int_{\Omega}\hspace{-0.8mm}\nabla
p(x).(x-a)|\nabla u(x)|^{2}dx\\[\bigskipamount]
&-&\displaystyle\frac{1}{2}\int_{\partial\Omega}p(x) (x-a).\nu
|\frac{\partial u}{\partial \nu}|^{2}dx,
\end{array}\label{poroge1}\end{equation} where $\nu$ denotes the outward
normal to $\partial\Omega$.\\
Combining (\ref{porog1}), (\ref{porog2}) and (\ref{poroge1}), we
write
\begin{equation}
\begin{array}{lll}
-\frac{n-2}{2}\displaystyle\int_{\Omega}p(x) |\nabla
u(x)|^{2}dx-\frac{1}{2}\int_{\Omega}\nabla p(x).(x-a)|\nabla
u(x)|^{2}dx=\\[\bigskipamount]
\hspace{70mm}\displaystyle-\frac{n-2}{2}\int_{\Omega}|u(x)|^{q}dx-\frac{n}{2}\lambda\int_{\Omega}|u(x)|^{2}dx.
\label{eqq1}
\end{array}
\end{equation}
On the other hand, we multiply (\ref{eqhy1}) by $\frac{n-2}{2} u$
and we integrate by parts, we get
\begin{equation}
\frac{n-2}{2}\int_{\Omega}p(x) |\nabla
u(x)|^{2}dx=\frac{n-2}{2}\int_{\Omega}|u(x)|^{q}dx+\frac{n-2}{2}\lambda\int_{\Omega}|u(x)|^{2}dx.
\label{eqq2}
\end{equation}
Combining (\ref{eqq1}) and (\ref{eqq2}), we obtain
$$
\lambda \int_{\Omega}|u(x)|^{2}dx-\frac{1}{2}\int_{\Omega}\nabla
p(x).(x-a)|\nabla u(x)
|^{2}dx-\frac{1}{2}\int_{\partial\Omega}p(x)|\frac{\partial
u}{\partial \nu}|^{2}(x-a).\nu dx =0.
$$
If $\Omega$ is starshaped about $a$, then $(x-a).\nu >0$ on
$\partial\Omega$, and
$$
\lambda\int_{\Omega}|u(x)|^{2}dx-\frac{1}{2}\int_{\Omega}\nabla
p(x).(x-a)| \nabla u(x)|^{2}dx> 0.
$$
It follows that
\begin{equation*}
\lambda>\frac{1}{2}\frac{\displaystyle\int_{\Omega}\nabla
p(x).(x-a)|\nabla u(x)|^{2}dx}{\displaystyle\int_{\Omega}|u|^{2}dx}
\end{equation*}
and we obtain the desired result.
\end{proof}
%%%%%%%%%%%%%%%%%%%
%%%%%%%%%%%%%%%
\section{Existence of solutions}
\medskip
Let $\Omega\in\R^{n}$, $n\ge {3}$ be a bounded domain. In this
section, we show that (\ref{eqhy1}) possesses a solution of lower
energy less than $p_{_{0}}S$. We will use a minimization
technique.\\
Set \ba
Q_{\lambda}(u)=\frac{\int_{\Omega}{p(x)}\vert\nabla{u(x)}\vert^{2}dx-\lambda\int_{\Omega}\vert
u(x)\vert^{2}dx}{\parallel{u}\parallel^{2}_{q}}\label{fonct} \ea the
functional associated to (\ref{eqhy1}).\\We define \ba
S_{\lambda}(p)=\inf_{{u}\in
{H}^{1}_{0}(\Omega),u\not=0}Q_{\lambda}(u).\label{eqhy8}\ea Let us
remark that\basn S_{\lambda}(p)=\inf_{{u}\in
{H}^{1}_{0}(\Omega),\|u\|_{q}=1}\int_{\Omega}{p(x)}\vert\nabla{u(x)}\vert^{2}dx-\lambda\int_{\Omega}\vert
u(x)\vert^{2}dx.\easn The method used for the proof of Theorem
{\ref{thhy2}} is the following : First we show that
$S_{\lambda}(p)<p_{_{0}}S$, we then prove that the infimum
$S_{\lambda}(p)$ is achieved.\\We have the following result
\begin{lemma}~\\
If $S_{\lambda}(p)<p_{_{0}} S$ for some $\lambda>0$, then the
infimum in (\ref{eqhy8}) is achieved. \label{lmhy3}
\end{lemma}
\begin{proof} Let $\{u_{j}\}\subset H_{0}^{1}(\Omega)$ be a minimizing sequence
for (\ref{eqhy8}) that is,
\begin{eqnarray}
\|u_{j}\|_{q}=1, \label{eqhy9}
\end{eqnarray}
\ba \int_{\Omega}p(x)|\nabla
u_{j}(x)|^{2}dx-\lambda\int_{\Omega}|u_{j}(x)|^{2}dx=S_{\lambda}(p)+o(1)\textrm{\,\,
as $j\rightarrow \infty$.} \label{eqhy10} \ea The sequence $u_{j}$
is bounded in $H_{0}^{1}(\Omega)$. Indeed, from (\ref{eqhy10}), we
have $$\int_{\Omega}p(x)|\nabla
u_{j}(x)|^{2}dx=S_{\lambda}(p)+\lambda\int_{\Omega}|u_{j}(x)|^{2}dx+o(1).$$
Using the embedding  of $L^{q}(\Omega)$ into $L^{2}(\Omega)$, there
exists a positive constant $C_{1}$ such that
$$\int_{\Omega}p(x)|\nabla
u_{j}(x)|^{2}dx\leq S_{\lambda}(p)+\lambda
\,C_{1}\|u_{j}\|_{q}^{2}+o(1).$$ Using the fact that
$$\|u_{j}\|_{q}=1,$$ we obtain
$$\int_{\Omega}p(x)|\nabla
u_{j}(x)|^{2}dx\leq S_{\lambda}(p)+\lambda\,C_{1}+o(1).$$ Since
$0<p_{_{0}}\leq p(x)$ for every $x\in \Omega$, we deduce
$$\int_{\Omega}|\nabla
u_{j}(x)|^{2}dx\leq\frac{S_{\lambda}(p)+\lambda\,C_{1}}{p_{_{0}}}+o(1).$$
This gives the desired result.\\Since $\{u_{j}\}$ is bounded in
$H_{0}^{1}(\Omega)$ we may extract a subsequence still denoted by
$u_{j}$, such that \basn u_{j}\rightharpoonup u &\textrm{weakly in
$H_{0}^{1}(\Omega)$},\\[\medskipamount]
u_{j}\rightarrow u &\textrm{strongly in $L^{2}(\Omega)$},\\[\medskipamount]
u_{j}\rightarrow u&\textrm{a.e. on $\Omega$}, \easn with
$\|u\|_{q}\leq {1}$. Set $v_{j}=u_{j}-u$, so that \basn
v_{j}\rightharpoonup {0}&\textrm{weakly in
$H_{0}^{1}(\Omega)$}\\[\medskipamount]
v_{j}\rightarrow 0 &\textrm{strongly in $L^{2}(\Omega)$},\\[\medskipamount]
v_{j}\rightarrow{0}&\textrm{a.e. on $\Omega$}. \easn Using
(\ref{eqhy9}), the definition of $S$ and the fact that
 $\displaystyle\min_{\bar{\Omega}}p(x)=p_{_{0}}>0$, we have
 $$\int_{\Omega}p(x)|\nabla u_{j}(x)|^{2}dx \geq p_{_{0}} S.$$
  From
(\ref{eqhy10}) it follows that $\lambda \|u\|_{2}^{2}\geq p_{_{0}}
S-S_{\lambda}(p)>0$ and therefore $u\not=0$. Using again
(\ref{eqhy10}) we obtain
\begin{equation} \int_{\Omega}p(x)|\nabla u(x)|^{2}dx+\int_{\Omega}p(x)|\nabla
v_{j}(x)|^{2}dx-\lambda\int_{\Omega}|u(x)|^{2}dx=S_{\lambda}(p)+o(1),
\label{eqhy11} \end{equation}
 since $v_{j}\rightharpoonup {0}$ weakly in
$H_{0}^{1}(\Omega)$. On the other hand, it follows from a result of
\cite{bl} that
$$
\|u+v_{j}\|_{q}^{q}=\|u\|_{q}^{q}+\|v_{j}\|_{q}^{q}+o(1),
$$
(which holds since $v_{j}$ is bounded in $L^{q}$ and
$v_{j}\rightarrow 0$ a.e.). Thus, by (\ref{eqhy9}), we have
$$
1=\|u\|_{q}^{q}+\|v_{j}\|_{q}^{q}+o(1)
$$
and therefore
$$
1\leq\|u\|_{q}^{2}+\|v_{j}\|_{q}^{2}+o(1),
$$
which leads to \ba 1\leq\|u\|_{q}^{q}+\frac{1}{p_{_{0}}
S}\int_{\Omega}p(x)|\nabla v_{j}(x)|^{2}dx+o(1). \label{eqhy12} \ea
We distinguish two cases:\\
(a) $S_{\lambda}(p)>0$, which corresponds to $0<\lambda<\lambda_{1}^{div}$,\\[\medskipamount]
(b) $S_{\lambda}(p)\leq0$, which corresponds to $\lambda \geq \lambda_{1}^{div}$.\\
In case (a) we deduce from (\ref{eqhy12}) that \ba
S_{\lambda}(p)\leq
S_{\lambda}(p)\|u\|_{q}^{2}+(\frac{S_{\lambda}(p)}{p_{_{0}}
S})\int_{\Omega}p(x)|\nabla v_{j}(x)|^{2}dx+o(1). \label{eqhy13} \ea
Combining (\ref{eqhy11}) and (\ref{eqhy13}) we obtain
\begin{equation*}
\begin{array}{ll}
\int_{\Omega}p(x)|\nabla u(x)|^{2}-\lambda
|u(x)|^{2}dx+\int_{\Omega}p(x)|\nabla v_{j}(x)|^{2}dx\leq
S_{\lambda}(p)\|u\|_{q}^{2}\\[\bigskipamount]
\hspace{80mm}+\displaystyle(\frac{S_{\lambda}(p)}{p_{_{0}}
S})\int_{\Omega}p(x)|\nabla v_{j}(x)|^{2}dx+o(1).
\end{array}
\end{equation*}
Thus
\begin{equation*}
\begin{array}{ll}
\int_{\Omega}p(x)|\nabla
u(x)|^{2}dx-\lambda\int_{\Omega}|u(x)|^{2}dx&\leq
S_{\lambda}(p)\|u\|_{q}^{2}\\[\bigskipamount]
&+\displaystyle\left[\frac{S_{\lambda}(p)}{p_{_{0}}
S}-1\right]\int_{\Omega}p(x)|\nabla v_{j}(x)|^{2}dx+o(1).
\end{array}
\end{equation*}
Since $S_{\lambda}(p)<p_{_{0}} S$, we deduce \ba
\int_{\Omega}p(x)|\nabla u(x)|^{2}dx-\lambda
\int_{\Omega}|u(x)|^{2}dx\leq S_{\lambda}(p)\|u\|_{q}^{2},
\label{eqhy14} \ea this means that $u$ is a minimum of
$S_{\lambda}(p)$.\\In case (b), since $\|u\|_{q}^{2}\leq{1}$, we
have $S_{\lambda}(p)\leq S_{\lambda}(p)\|u\|_{q}^{2}$. Again, we
deduce (\ref{eqhy14}) from (\ref{eqhy11}). This concludes the proof
of Lemma~{\ref{lmhy3}}.
\end{proof}
To prove assertion 1) and 2) of Theorem \ref{thhy2} (case $k\geq
2$), we need the following
\begin{lemma}~\\
a) For $n\geq {4}$, we have
$$
S_{\lambda}(p)<p_{_{0}} S\,\, \textrm{for all\, $\lambda>0$\, and
for \,$k>2$.}
$$
b) For $n=4$ and $k=2$, we have
$$
S_{\lambda}(p)<p_{_{0}} S\,\, \textrm{for all\,
$\lambda>4\beta_{2}$.}
$$
c) For $n\geq5$ and $k=2$, we have
$$
S_{\lambda}(p)<p_{_{0}} S\,\, \textrm{for all\,
$\lambda>\frac{(n-2)n(n+2)}{4(n-1)}\beta_{2}$.}
%\frac{(n-2)^{2}}\beta_{2}\int_{\R^{n}}\frac{|x|^{4}}{(1+|x|^{2})^{n}}dx}{B} $
$$
d) For $n=3$ and $k\ge 2$, we have
$$
S_{\lambda}(p)<p_{_{0}}S\quad \textrm{for all\,
$\lambda>\gamma(k)$\, where $\gamma(k)$ is a positive constant.}$$
\label{lmhy2}
\end{lemma}
\begin{proof} We shall estimate the ratio $Q_{\lambda}(u)$ defined in (\ref{fonct}), with $u=u_{{a},\varepsilon}$.\\We claim that, as $\varepsilon\rightarrow 0$, we have
\begin{equation}
\begin{array}{ll}
\varepsilon^{\frac{n-2}{2}}\hspace{-3mm}
\displaystyle\int_{\Omega}p(x)\vert{\nabla
u_{{a},\varepsilon}(x)}\vert^{2}dx\leq\\[\bigskipamount] \left\{\begin{array}{llll}
\hspace{-1.5mm}p_{_{0}}K_{1}+O(\varepsilon^{\frac{n-2}{2}})&\hspace{-1.8mm}\textrm{if
$\left\{\begin{array}{lll}n\ge4\textrm{\quad and}\\
n-2<k,\end{array}\right.$}\\[\medskipamount]
\hspace{-1.5mm}p_{_{0}}K_{1}+A_{_{k}}\varepsilon^{\frac{k}{2}}+o(\varepsilon^{\frac{k}{2}})&\hspace{-1.8mm}\textrm{if $\left\{\begin{array}{lll}n\ge{4}\textrm{\quad and}\\n-2>k,\end{array}\right.$}\\[\medskipamount]
\hspace{-1.5mm}p_{_{0}}K_{1}+\displaystyle\frac{(n-2)^{2}(\beta_{n-2}+M)
\omega_{n}\varepsilon^{\frac{n-2}{2}}|\log\varepsilon|}{2}+o(\varepsilon^{\frac{n-2}{2}}|\log\varepsilon|)&\hspace{-3.5mm}\textrm{
if $\left\{\begin{array}{ll}n>4\textrm{\quad and}\\k=n-2,
\end{array}\right.$}\\[\medskipamount]
\hspace{-1.5mm}p_{_{0}}K_{1}+2 \beta_{2} \omega_{_{4}}
\varepsilon|\log\varepsilon|+o(\varepsilon|\log\varepsilon|)&\hspace{-1.8mm}\textrm{if
$\left\{\begin{array}{lll}n=4\textrm{\quad and}\\
k=2,\end{array}\right.$}
\end{array}
\right. \label{eqhy15}
\end{array}
\end{equation}
with $K_{1}=(n-2)^{2}\int_{\R^{n}}\frac{\vert y\vert^{2}}{(1+\vert
y\vert^{2})^{n}}dy$, $s=\min(\frac{k}{2},\frac{n-2}{2})$,
$A_{_{k}}=(n-2)^{2} \beta_{k}
\int_{\R^{n}}\frac{|x|^{k+2}}{(1+|x|^{2})^{n}}dx$ and $M$ is a
positive constant.
\\[2ex]
\,\,{\bf{Verification of~(\ref{eqhy15})}}\\
{\bf1. Case $n\ge 4$ and $k>0$, with $k\not=2$ if $n=4$.}\\
We have\\
 \basn
 %\begin{split}
 \int_{\Omega}p(x)\vert{\nabla{
u_{{a},\varepsilon}(x)}\vert^{2}}dx&=&\int_{\Omega}\frac{
p(x)\vert\nabla\zeta(x)\vert^{2}}{(\varepsilon +
\vert{x-a}\vert^{2})^{n-2}}dx+(n-2)^{2}\int_{\Omega}\frac{p(x)\vert\zeta(x)\vert^{2}\vert{x-a}\vert^{2}}{(\varepsilon
+ \vert{x-a}\vert^{2})^{n}}dx\\[\bigskipamount]
&-&2(n-2)\int_{\Omega}\frac{p(x)\zeta(x)\nabla{\zeta(x)}(x-a)}{(\varepsilon
+ \vert{x-a}\vert^{2})^{n-1}}dx.
 %\end{split}
 \easn
 Since $\zeta\equiv{1}$ on a neighborhood of $a$, we assume that
 $\varphi\equiv 1$ on $B(a,l)$ with $l$ is a small positive constant.
 Therefore we get $|\nabla \varphi|^{2}\equiv 0$ on $B(a,l)$ and
 $\nabla\varphi(x).(x-a)=0$ on $B(a,l)$.\\Thus, we obtain

\begin{equation}
\begin{array}{lll}
 \hspace{-0.8mm}\displaystyle\int_{\Omega}p(x)\vert{\nabla{
u_{{a},\varepsilon}(x)}\vert^{2}}dx\hspace{-2mm}&=&\hspace{-4mm}\displaystyle\int_{\Omega\setminus
B(a,l)}\hspace{-1mm}\frac{
p(x)\vert\nabla\zeta(x)\vert^{2}}{(\varepsilon +
\vert{x-a}\vert^{2})^{n-2}}dx\hspace{-0.3mm}+\hspace{-0.3mm}(n-2)^{2}\hspace{-1mm}\int_{\Omega}\hspace{-1mm}\frac{p(x)\vert\zeta(x)\vert^{2}\vert{x-a}\vert^{2}}{(\varepsilon
+ \vert{x-a}\vert^{2})^{n}}dx\\[\bigskipamount]
&-&2(n-2)\displaystyle\int_{\Omega\setminus
B(a,l)}\frac{p(x)\zeta(x)\nabla{\zeta(x)}(x-a)}{(\varepsilon +
\vert{x-a}\vert^{2})^{n-1}}dx.
\end{array}
\label{domine1}
\end{equation}
Therefore, applying the Dominated Convergence Theorem,
(\ref{domine1}) becomes

\begin{equation*}
\displaystyle\int_{\Omega}p(x)\vert{\nabla{
u_{{a},\varepsilon}(x)}\vert^{2}}dx=(n-2)^{2}\displaystyle\int_{\Omega}\frac{p(x)\vert\zeta(x)\vert^{2}\vert{x-a}\vert^{2}}{(\varepsilon
+ \vert{x-a}\vert^{2})^{n}}dx+O(1).
\end{equation*}
Using (\ref{eqhy0}), a direct computation gives \basn
\varepsilon^{\frac{n-2}{2}}\int_{\Omega}p(x)\vert{\nabla
u_{{a},\varepsilon}(x)}\vert^{2}
dx&=&(n-2)^{2}p_{_{0}}\varepsilon^{\frac{n-2}{2}}\int_{\Omega}\frac{|
x-a|^{2}}{(\varepsilon +
|{x-a}|^{2})^{n}}dx\\[\medskipamount]
 &+&(n-2)^{2}\varepsilon^{\frac{n-2}{2}}
\beta_{k}\int_{\Omega}\frac{\vert x-a\vert^{k+2}}{(\varepsilon
+|{x-a}|^{2})^{n}}dx\\[\medskipamount]&+&(n-2)^{2}\varepsilon^{\frac{n-2}{2}}\int_{\Omega}\frac{|
x-a|^{k+2}\theta(x)}{(\varepsilon+|x-a|^{2})^{n}}dx\\[\medskipamount]
&+&(n-2)^{2}\varepsilon^{\frac{n-2}{2}}\int_{\Omega}\frac{|
x-a|^{k+2}(\beta_{k}+\theta(x))(|\zeta(x)|^{2}-{1})}{(\varepsilon +
|{x-a}|^{2})^{n}}dx\\[\medskipamount]
&+&O(\varepsilon^{\frac{n-2}{2}}). \easn Using again the definition
of $\zeta$, and applying the Dominated Convergence Theorem, we
obtain \basn
\varepsilon^{\frac{n-2}{2}}\int_{\Omega}p(x)\vert{\nabla
u_{{a},\varepsilon}(x)}\vert^{2} dx&\hspace{-1mm}=\hspace{-1mm}&
(n-2)^{2}p_{_{0}}\varepsilon^{\frac{n-2}{2}}\hspace{-3mm}\int_{\Omega}\hspace{-1mm}\frac{|
x-a|^{2}}{(\varepsilon
\hspace{-0.5mm}+\hspace{-0.5mm}|{x-a}|^{2})^{n}}dx\\[\medskipamount]&+&\hspace{-0.8mm}(n-2)^{2}\varepsilon^{\frac{n-2}{2}}
\beta_{k}\hspace{-1.5mm}\int_{\Omega}\hspace{-1.5mm}\frac{\vert
x-a\vert^{k+2}}{(\varepsilon +
\vert{x-a}\vert^{2})^{n}}\\[\medskipamount]
&\hspace{-1mm}+\hspace{-1mm}&(n-2)^{2}\varepsilon^{\frac{n-2}{2}}\int_{\Omega}\frac{|
x-a|^{k+2}\theta(x)}{(\varepsilon+|x-a|^{2})^{n}}dx+O(\varepsilon^{\frac{n-2}{2}}).
\easn Here we will consider the following three subcases:\\{\bf 1.1.
If $n-2>k$},
 \begin{equation*}
 \begin{array}{llll}
 \displaystyle\varepsilon^{\frac{n-2}{2}}\int_{\Omega}p(x)\vert{\nabla
u_{{a},\varepsilon}(x)}\vert^{2} dx\\[\medskipamount]
\hspace{26mm}=\displaystyle p_{_{0}}(n-2)^{2}\varepsilon^\frac{n-2}{2}\left\lbrack\int_{\R^{n}}\frac{|x-a|^{2}}{(\varepsilon+|x-a|^{2})^{n}}dx-\int_{\R^{n}\setminus\Omega}\frac{|x-a|^{2}}{(\varepsilon+|x-a|^{2})^{n}}dx\right\rbrack\\[\medskipamount]
\\[\medskipamount]
\hspace{26mm}=\displaystyle(n-2)^{2}\varepsilon^\frac{n-2}{2}\left\lbrack\hspace{-0.8mm}\int_{\R^{n}}\hspace{-1.5mm}\frac{|x-a|^{k+2}(\beta_{k}+\theta(x))}{(\varepsilon+|
x-a|^{2})^{n}}-\hspace{-0.8mm}\int_{\R^{n}\setminus\Omega}\hspace{-5mm}\frac{\vert
x-a\vert^{k+2}(\beta_{k}+\theta(x))}{(\varepsilon+\vert
x-a\vert^{2})^{n}}\right\rbrack\\[\medskipamount]
\hspace{26mm}=O(\varepsilon^{\frac{n-2}{2}}).
\end{array}\end{equation*} Using a simple change of variable and applying the Dominated
Convergence Theorem, we find \basn
\varepsilon^{\frac{n-2}{2}}\hspace{-1mm}
\int_{\Omega}\hspace{-2mm}p(x)\vert{\nabla
u_{{a},\varepsilon}(x)}\vert^{2}
dx&=&\displaystyle\hspace{-0.5mm}p_{_{0}}\hspace{-1.8mm}\int_{\R^{n}}\hspace{-2.1mm}\frac{|y|^{2}}{(1+|y|^{2})^{n}}\hspace{-0.4mm}+\hspace{-0.4mm}(n-2)^{2}\varepsilon^{\frac{k}{2}}\hspace{-1.5mm}\int_{\R^{n}}\hspace{-2.9mm}\frac{|y|^{k+2}(\beta_{k}+\theta(a+\varepsilon^\frac{1}{2}
y))}{(1+|y|^{2})^{n}}dy\\[\medskipamount]
&+&o(\varepsilon^{\frac{k}{2}}). \easn The fact that $\theta(x)$
tends to $0$ when $x$ tends to $a$ gives that \basn
\varepsilon^{\frac{n-2}{2}} \int_{\Omega}p(x)\vert{\nabla
u_{{a},\varepsilon}(x)}\vert^{2}dx=
p_{_{0}}K_{1}+A_{_{k}}\varepsilon^{\frac{k}{2}}+o(\varepsilon^{\frac{k}{2}}),
\easn with $K_{1}=(n-2)^{2}\int_{\R^{n}}\frac{\vert
y\vert^{2}}{(1+\vert y\vert^{2})^{n}}dy$ and
$A_{_{k}}=\beta_{k}\int_{\R^{n}}\frac{\vert y\vert^{k+2}}{(1+\vert
y\vert^{2})^{n}}dy$. \\~\\{\bf 1.2. If ${n-2}<k$}, \basn
\varepsilon^{\frac{n-2}{2}} \int_{\Omega}p(x)\vert{\nabla
u_{{a},\varepsilon}(x)}\vert^{2}dx
&=&p_{_{0}}K_{1}+(n-2)^{2}\beta_{k}\varepsilon^{\frac{n-2}{2}}\int_{\Omega}\frac{\vert
x-a\vert^{k+2}}{(\varepsilon+\vert
x-a\vert^{2})^{n}}dx\\[\medskipamount]
&+&(n-2)^{2}\varepsilon^{\frac{n-2}{2}}\int_{\Omega}\frac{\vert
x-a\vert^{k+2}\theta(x)}{(\varepsilon+\vert
x-a\vert^{2})^{n}}dx+O(\varepsilon^{\frac{n-2}{2}}).\easn Since
$\Omega$ is a bounded domain, there exists some positive constant
$R$ such that $\Omega \subset B(a,R)$ and thus
\begin{equation*}\begin{array}{llll}\displaystyle \varepsilon^{\frac{n-2}{2}} \int_{\Omega}p(x)|{\nabla
u_{{a},\varepsilon}(x)}|^{2} dx\hspace{-1mm}=p_{_{0}}K_{1}+O(\varepsilon^{\frac{n-2}{2}})\\[\bigskipamount]\displaystyle
\hspace{4mm}+(n-2)^{2}\varepsilon^{\frac{n-2}{2}}\left\lbrack
\int_{B(a,R)}\hspace{-1.5mm}\frac{|
x-a|^{k+2}(\beta_{k}+\theta(x))}{(\varepsilon+|
x-a|^{2})^{n}}dx-\hspace{-1mm}\int_{B(a,R)\setminus\Omega}\hspace{-2.3mm}\frac{|
x-a|^{k+2}(\beta_{k}+\theta(x)}{(\varepsilon+|
x-a|^{2})^{n}}dx\right\rbrack.\end{array}\end{equation*} By a simple
change of variable, we get \basn \varepsilon^{\frac{n-2}{2}}
\int_{\Omega}\hspace{-1mm}p(x)\vert{\nabla
u_{{a},\varepsilon}(x)}\vert^{2}
dx\hspace{-0.5mm}&=&p_{_{0}}K_{1}+\hspace{-0.5mm}(n-2)^{2}\varepsilon^{\frac{n-2}{2}}\int_{B(0,R)}\hspace{-0.5mm}\frac{|y|^{k+2}(\beta_{k}+\theta(a+ y))}{(\varepsilon+|y|^{2})^{n}}dy\\[\medskipamount]
&+&O(\varepsilon^{\frac{n-2}{2}}).\easn Using the definition of
$\theta$ given by (\ref{eqhy0}), there exists a positive constant
$M$ such that \basn \varepsilon^{\frac{n-2}{2}}
\int_{\Omega}p(x)\vert{\nabla u_{{a},\varepsilon}(x)}\vert^{2}
dx&\leq&
p_{_{0}}K_{1}+(n-2)^{2}\varepsilon^{\frac{n-2}{2}}(\beta_{k}+M)\int_{B(0,R)}\frac{|y|^{k+2}}{(\varepsilon+|y|^{2})^{n}}dy\\[\medskipamount]
&+&O(\varepsilon^{\frac{n-2}{2}}). \easn Applying the Dominated
Convergence Theorem we deduce that \basn \varepsilon^{\frac{n-2}{2}}
\int_{\Omega}p(x)|{\nabla u_{{a},\varepsilon}(x)}|^{2}dx\leq
p_{_{0}}K_{1}+O(\varepsilon^{\frac{n-2}{2}})
\easn and this completes the proof of (\ref{eqhy15}) in this case.\\
{\bf 1.2. If $k={n-2}$},\basn \varepsilon^{\frac{n-2}{2}}
\int_{\Omega}p(x)\vert{\nabla u_{{a},\varepsilon}(x)}\vert^{2}dx
&=&p_{_{0}}K_{1}+(n-2)^{2}\beta_{n-2}\varepsilon^{\frac{n-2}{2}}\int_{\Omega}\frac{\vert
x-a\vert^{n}}{(\varepsilon+\vert
x-a\vert^{2})^{n}}dx\\
&+&(n-2)^{2}\varepsilon^{\frac{n-2}{2}}\int_{\Omega}\frac{\vert
x-a\vert^{n}\theta(x)}{(\varepsilon+\vert
x-a\vert^{2})^{n}}dx+O(\varepsilon^{\frac{n-2}{2}}). \easn Since
$\Omega$ is a bounded domain, there exists some positive constant
$R$ such that $\Omega \subset B(a,R)$ and thus
\begin{equation*}\begin{array}{lll} \varepsilon^{\frac{n-2}{2}}
\displaystyle\int_{\Omega}p(x)\vert{\nabla
u_{{a},\varepsilon}(x)}\vert^{2}
dx=p_{_{0}}K_{1}+O(\varepsilon^{\frac{n-2}{2}})\\[\bigskipamount]
+(n-2)^{2}\varepsilon^{\frac{n-2}{2}}\displaystyle\left\lbrack\int_{B(a,R)}\hspace{-1.8mm}\frac{\vert
x-a\vert^{n}(\beta_{n-2}+\theta(x))}{(\varepsilon+\vert
x-a\vert^{2})^{n}}dx\hspace{-0.7mm}-\hspace{-0.7mm}\int_{B(a,R)\setminus\Omega}\hspace{-3.5mm}\frac{\vert
x-a\vert^{n}(\beta_{n-2}+\theta(x))}{(\varepsilon+\vert
x-a\vert^{2})^{n}}dx\right\rbrack.\end{array}\end{equation*} Hence
\basn \varepsilon^{\frac{n-2}{2}} \int_{\Omega}p(x)\vert{\nabla
u_{{a},\varepsilon}(x)}\vert^{2}
dx&=&p_{_{0}}K_{1}+(n-2)^{2}\varepsilon^{\frac{n-2}{2}}\int_{B(a,R)}\frac{|x-a|^{n}(\beta_{n-2}+\theta(x))}{(\varepsilon+\vert
x-a\vert^{2})^{n}}dx\\[\medskipamount]
&+&O(\varepsilon^{\frac{n-2}{2}}).\easn Using the definition of
$\theta$ given by (\ref{eqhy0}), there exists a positive constant
$M$ such that \begin{equation}\begin{array}{ll}
\displaystyle\varepsilon^{\frac{n-2}{2}}\hspace{-0.5mm}
\int_{\Omega}\hspace{-0.5mm}p(x)\vert{\nabla
u_{{a},\varepsilon}(x)}\vert^{2}
dx\hspace{-0.5mm}&\hspace{-1mm}\leq\hspace{-0.5mm}\displaystyle
p_{_{0}}K_{1}+\hspace{-0.5mm}(n-2)^{2}(\beta_{n-2}+M)\varepsilon^{\frac{n-2}{2}}\hspace{-0.3mm}\int_{B(a,R)}\hspace{-1.5mm}\frac{|x-a|^{n}}{(\varepsilon+\vert
x-a\vert^{2})^{n}}dx\\[\medskipamount]
&+O(\varepsilon^{\frac{n-2}{2}})\label{ehy1}.\end{array}\end{equation}
On the other hand, an easy computation gives \begin{equation*}
\begin{array}{ll}
\varepsilon^{\frac{n-2}{2}}\displaystyle\int_{B(a,R)}\frac{|
x-a|^{n}}{(\varepsilon+|
x-a|^{2})^{n}}dx&=\displaystyle\omega_{n}\varepsilon^{\frac{n-2}{2}}\int_{0}^{R}\frac{r^{2n-1}}{(\varepsilon+r^{2})^{n}}dr\\[\bigskipamount]
&=\displaystyle\frac{\omega_{n}}{2n}\varepsilon^{\frac{n-2}{2}}\int_{0}^{R}\frac{((\varepsilon+r^{2})^{n})'}{(\varepsilon+r^{2})^{n}}dr+O(\varepsilon^{\frac{n-2}{2}})
\end{array}
\end{equation*}
and
\begin{equation} \varepsilon^{\frac{n-2}{2}}\int_{B(a,R)}\frac{|
x-a|^{n}}{(\varepsilon+|
x-a|^{2})^{n}}dx=\frac{\omega_{n}}{2}\varepsilon^{\frac{n-2}{2}}|\log\varepsilon|+o(\varepsilon^{\frac{n-2}{2}}|\log\varepsilon|).
\label{ehy2}
\end{equation}
Inserting (\ref{ehy2}) into (\ref{ehy1}) we obtain
\begin{equation*}\varepsilon^{\frac{n-2}{2}}
\int_{\Omega}p(x)\vert{\nabla u_{{a},\varepsilon}(x)}\vert^{2}
dx\leq p_{_{0}}K_{1}+\frac{(n-2)^{2}(\beta_{n-2}+M)
\omega_{n}}{2}\varepsilon^{\frac{n-2}{2}}|\log\varepsilon|+o(\varepsilon^{\frac{n-2}{2}}|\log\varepsilon|).
\end{equation*}
{\bf 2)Case $n=4$ and $k=2$}.\\As we have announced in the
introduction, we assume in this case the following additional
condition on $\theta$:
$\int_{\B(a,1)}\frac{\theta(x)}{|x-a|^{4}}dx<\infty$. We have
\begin{eqnarray*}
\int_{\Omega}p(x)|\nabla
u_{a,\varepsilon}|^{2}dx&=&\int_{\Omega}\frac{p(x)|\nabla
\zeta(x)|^{2}}{(\varepsilon+|x-a|^{2})^{2}}dx+4\int_{\Omega}\frac{p(x)|\zeta(x)|^{2}|x-a|^{2}}{(\varepsilon+|x-a|^{2})^{4}}dx\\[\medskipamount]
&-&4\int_{\Omega}\frac{p(x)\zeta(x)\nabla\zeta(x)
(x-a)}{(\varepsilon+|x-a|^{2})^{3}}dx.\\[\medskipamount]
\end{eqnarray*}
Using ({\ref{eqhy0}}) and the fact that $\zeta\equiv 1$ near $a$, it
follows that
\begin{eqnarray*}
\int_{\Omega}p(x)|\nabla u_{a,\varepsilon}|^{2}dx&=&4p_{_{0}}\int_{\Omega}\frac{|\zeta(x)|^{2}|x-a|^{2}}{(\varepsilon+|x-a|^{2})^{4}}dx+4\beta_{2}\int_{\Omega}\frac{|\zeta(x)|^{2}|x-a|^{4}}{(\varepsilon+|x-a|^{2})^{4}}dx\\[\medskipamount]
&+&4\int_{\Omega}\frac{|x-a|^{4}\theta(x)}{(\varepsilon+|x-a|^{2})^{4}}dx+O(1),\\[\medskipamount]
&=&\frac{4p_{_{0}}}{\varepsilon}\int_{\R^{n}}\frac{|y|^{2}}{(1+|y|^{2})^{4}}dy+4\int_{\Omega}\frac{|x-a|^{4}(\beta_{k}+\theta(x))}{(\varepsilon+|x-a|^{2})^{4}}dx+O(1).
\end{eqnarray*}Since $\int_{B(a,1)}\frac{\theta(x)}{|x-a|^{4}}dx<\infty$, we obtain \basn
\int_{\Omega}\frac{|x-a|^{4}\theta(x)}{(\varepsilon+|
x-a|^{2})^{4}}dx&=&\int_{\Omega}\frac{\theta(x)}{|x-a|^{4}}dx+o(1)\\[\medskipamount]
&=&O(1).\easn Consequently \basn \int_{\Omega}p(x)|\nabla
u_{a,\varepsilon}|^{2}dx=\frac{4p_{_{0}}}{\varepsilon}\int_{\R^{n}}\frac{|y|^{2}}{(1+|y|^{2})^{4}}dy+4\beta_{k}\int_{\Omega}\frac{|x-a|^{4}}{(\varepsilon+|x-a|^{2})^{4}}dx+O(1).
\easn Let $R_{i}>0$, $i=1,2$ such that
$$
\int_{|x-a|\leq
R_{1}}\frac{|x-a|^{4}}{(\varepsilon+|x-a|^{2})^{4}}dx\leq
\int_{\Omega}\frac{|x-a|^{4}}{(\varepsilon+|x-a|^{2})^{4}}dx \leq
\int_{|x-a|\leq
R_{2}}\frac{|x-a|^{4}}{(\varepsilon+|x-a|^{2})^{4}}dx.
$$
We see that
\begin{eqnarray*}
\int_{|x-a|\leq
R}\frac{|x-a|^{4}}{(\varepsilon+|x-a|^{2})^{4}}dx&=&\omega_{_{4}}\int_{0}^{R}\frac{r^{7}}{(\varepsilon+r^{2})^{4}}dr,\\[\medskipamount]
&=&\frac{1}{8}\omega_{_{4}}\int_{0}^{R}\frac{((\varepsilon+r^{2})^{4})'}{(\varepsilon+r^{2})^{4}}dr-\omega_{_{4}}\int_{0}^{R}\frac{r\varepsilon^{3}+3r^{3}\varepsilon^{2}+3\varepsilon
r^{4}}{(\varepsilon+r^{2})^{4}}dr,\\[\medskipamount]
&=&\frac{1}{2}\omega_{_{4}}|\log
\varepsilon|-\omega_{_{4}}\int_{0}^{\frac{R}{\varepsilon^{\frac{1}{2}}}}\frac{t+3t^{3}+3t^{5}}{(1+t^{2})^{4}}dt + O(1),\\[\medskipamount] &=&\frac{1}{2}\omega_{_{4}}|\log \varepsilon|+O(1).
\end{eqnarray*}
Hence, we have
\begin{eqnarray*}
\int_{\Omega}p(x)|\nabla u_{a,\varepsilon}|^{2}dx=\frac{p_{_{0}}
K_{1}}{\varepsilon}+2\beta_{2}\omega_{_{4}}|\log \varepsilon|+O(1),
\end{eqnarray*}
where $K_{1}=\int_{\R^{n}}\frac{|y|^{2}}{(1+|y|^{2})^{4}}dy$. This completes the proof of~(\ref{eqhy15}).\\Let us come back to the proof of Lemma~\ref{lmhy2}.\\
It is convenient to rewrite (\ref{eqhy15}) as
\begin{equation}
\varepsilon^{\frac{n-2}{2}}\int_{\Omega}p(x)|\nabla
u_{a,\varepsilon}|^{2}dx\leq\left\{
\begin{array}{lll}
p_{_{0}}K_{1}+o(\varepsilon)
&\textrm{if  $n\geq{5}$, and $k>2$ },\\[\medskipamount]
p_{_{0}}K_{1}+A_{2} \varepsilon+o(\varepsilon)
&\textrm{if  $n\geq{5}$, and $k=2$ },\\[\medskipamount]
p_{_{0}}K_{1} + A_{_{k}} \varepsilon^{\frac{k}{2}}
+o(\varepsilon^{\frac{k}{2}})
&\textrm{if  $n\geq{4}$, and $k< 2$ },\\[\medskipamount]
p_{_{0}}K_{1}+o(\varepsilon)&\textrm{ if $n={4}$, and $k>2$ },\\[\medskipamount]
p_{_{0}}K_{1}+2\omega_{_{4}}\beta_{2}\varepsilon|\log {\varepsilon
}| +o(\varepsilon|\log {\varepsilon }|)&\textrm{ if $n={4}$, and
$k=2$ }.
\end{array}
\right. \label{eqhy15'}
\end{equation}
Combining~(\ref{eqhy15'}), (\ref{eqhy3}) and (\ref{eqhy4}), we
obtain
\begin{equation}
S_{\lambda}(p)\leq Q_{\lambda}(u_{a,\varepsilon})\leq\left\{
\begin{array}{lll}
p_{_{0}}S-\lambda\frac{K_{3}}{K_{2}}\varepsilon+o(\varepsilon)
&\textrm{if  $n\geq{5}$, and $k>2$ },\\[\medskipamount]
p_{_{0}}S-(\lambda-C) \frac{K_{3}}{K_{2}} \varepsilon+o(\varepsilon)
&\textrm{if  $n\geq{5}$, and $k=2$ },\\[\medskipamount]
p_{_{0}}S + A_{_{k}} \varepsilon^{\frac{k}{2}}
+o(\varepsilon^{\frac{k}{2}})
&\textrm{if  $n\geq{4}$, and $k< 2$ },\\[\medskipamount]
p_{_{0}}S-\lambda \frac{\omega_{_{4}}}{2K_{2}}\varepsilon|\log
{\varepsilon }| +o(\varepsilon|\log
{\varepsilon }|)&\textrm{ if $n={4}$, and $k>2$ },\\[\medskipamount]
p_{_{0}}S-\frac{\omega_{_{4}}}{2 K_{2}}[\lambda
-4\beta_{2}]\varepsilon|\log {\varepsilon }| +o(\varepsilon|\log
{\varepsilon }|)&\textrm{ if $n={4}$, and $k=2$
},\\[\medskipamount]
\end{array}
\right. \label{eqhy16}
\end{equation}
with $C=\frac{A_{2}}{K_{3}}=\frac{\beta_{2} (n-2)n(n+2)}{4(n-1)}
$.\\~\\Assertions a), b) and c) of Lemma~\ref{lmhy2} follow directly
for
$\varepsilon$ small enough.\\
Now we prove d) of Lemma~\ref{lmhy2} (case $n=3$ and $k\geq 2$). We
will estimate the ratio
$$
Q_{\lambda}(u)=\frac{\int_{\Omega}p(x)|\nabla u|^{2}dx-\lambda
\|u\|_{2}^{2}}{\|u\|_{q}^{2}}
$$
with
$$
u(x)=u_{\varepsilon,a}(r)=\frac{\zeta(r)}{(\varepsilon+r^{2})^{\frac{1}{2}}},
\textrm{$r=|x|$, $\varepsilon>0$},
$$
where $\zeta$ is a fixed smooth function satisfying
 $0\leq\zeta\leq1$, $\zeta=1$ in
$\{x,\, |x-a|<\frac{R}{2}\}$ and $\zeta=0$ in $\{x,\, |x-a|\geq R\}$, where $R$ is a positive constant such that $B(a,R)\subset \Omega$.\\
We claim that, as $\varepsilon \rightarrow 0$,
\begin{equation}
\int p(x)|\nabla
u_{a,\varepsilon}(x)|^{2}dx=\frac{p_{_{0}}K_{1}}{\varepsilon^{\frac{1}{2}}}+\omega_{_{3}}\int_{0}^{R}(p_{0}+\beta_{k}r^{k})|\zeta'(r)|^{2}dr+\omega_{_{3}}
k\int_{0}^{R}|\zeta|^{2}r^{k-2}dr+o(1). \label{eq21}
\end{equation}
And from \cite{bn}, we already have\\
\begin{eqnarray}
\|\nabla
u_{a,\varepsilon}\|_{2}^{2}=\frac{K_{1}}{\varepsilon^{\frac{1}{2}}}+\omega_{3}\int_{0}^{R}|\zeta'(r)|^{2}dr+O(\varepsilon^{\frac{1}{2}}),
\label{eqb}
\end{eqnarray}
\begin{eqnarray}
\|u_{a,\varepsilon}\|_{6}^{2}=\frac{K_{2}}{\varepsilon^{\frac{1}{2}}}+O(\varepsilon^{\frac{1}{2}}),
\label{eq22}
\end{eqnarray}
\begin{eqnarray}
\|u_{a,\varepsilon}\|_{2}^{2}=\omega_{_{3}}
\int_{0}^{R}\zeta^{2}(r)dr+O(\varepsilon^{\frac{1}{2}}),
\label{eq23}
\end{eqnarray}
where $K_{1}$ and $K_{2}$ are positive constants such that
$\frac{K_{1}}{K_{2}}=S$ and $\omega_{_{3}}$ is the area of $S^{2}$
.\\~\\
{\bf Verification of~(\ref{eq21}).}\\
Using (\ref{eqhy0}), (\ref{eqb}) and the fact that $\zeta=0$ in
$\{x,\, |x-a|\geq R\}$, we write
\begin{eqnarray*}
\int p(x)|\nabla
u_{a,\varepsilon}(x)|^{2}dx&=&\frac{p_{_{0}}K_{1}}{\varepsilon^{\frac{1}{2}}}+\omega_{_{3}}
p_{_{0}}\int_{0}^{R}|\zeta'(r)|^{2}dr\\[\medskipamount]
&+&\omega_{_{3}}\beta_{k}\int_{0}^{R}\left[\frac{|\zeta'(r)|^{2}}{\varepsilon+r^{2}}-\frac{2
r\zeta(r)\zeta'(r)}{(\varepsilon+r^{2})^{2}}+\frac{r^{2}\zeta^{2}(r)}{(\varepsilon+r^{2})^{3}}\right]r^{k+2}dr\\[\medskipamount]
&+&O(\varepsilon^{\frac{1}{2}}).
\end{eqnarray*}
The fact that $\zeta=1$ in $\{x,\, |x-a|<\frac{R}{2}\}$,
$\zeta'(0)=0$ and $\zeta(R)=0$ gives
\begin{eqnarray*}
-2\int_{0}^{R}\frac{\zeta(r)\zeta'(r)
r^{k+3}}{(\varepsilon+r^{2})^{2}}dr=(k+3)\int_{0}^{R}\frac{|\zeta(r)|^{2}r^{k+2}}{(\varepsilon+r^{2})^{2}}dr-4\int_{0}^{R}\frac{|\zeta(r)|^{2}r^{k+4}}{(\varepsilon+r^{2})^{3}}dr.
\end{eqnarray*}
Consequently
\begin{eqnarray*}
\int_{\Omega} p(x)|\nabla
u_{a,\varepsilon}(x)|^{2}dx&=&\frac{p_{_{0}}K_{1}}{\varepsilon^{\frac{1}{2}}}+\omega_{_{3}}
p_{_{0}} \int_{0}^{R}|\zeta'(r)|^{2}dr+\omega_{_{3}}\beta_{k}
\int_{0}^{R}\frac{|\zeta'(r)|^{2}r^{k+2}}{\varepsilon+r^{2}}dr\\[\medskipamount]&-&\hspace{-1mm}3\omega_{_{3}}\beta_{k}\int_{0}^{R}\hspace{-0.5mm}\frac{|\zeta(r)|^{2}r^{k+4}}{(\varepsilon+r^{2})^{3}}dr
+(k+3)\omega_{_{3}}\beta_{k}\int_{0}^{R}\frac{|\zeta(r)|^{2}r^{k+2}}{(\varepsilon+r^{2})^{2}}dr\\[\medskipamount]
&+&O(\varepsilon^{\frac{1}{2}}).
\end{eqnarray*}
Applying the Dominated Convergence Theorem, we get the desired
result.\\Combining (\ref{eq21}), (\ref{eq22}) and (\ref{eq23}), we
obtain
\begin{eqnarray*}
Q_{\lambda}(u_{a,\varepsilon})\hspace{-1.8mm}&\hspace{-1.8mm}=\hspace{-1.8mm}&\hspace{-1.8mm}p_{_{0}}S+\omega_{_{3}}\hspace{-1mm}\left[\int_{0}^{R}\hspace{-1.5mm}(p_{_{0}}+\beta_{k}r^{k})|\zeta'(r)|^{2}dr\hspace{-1mm}+\hspace{-1mm}
k\beta_{k}\int_{0}^{R}\hspace{-1mm}|\zeta(r)|^{2}r^{k-2}dr\hspace{-1mm}-\hspace{-1mm}\lambda\hspace{-1.3mm}
\int_{0}^{R}\hspace{-2mm}\zeta^{2}(r)dr\right]\hspace{-1mm}\frac{\varepsilon^{\frac{1}{2}}}{K_{2}}\\[\medskipamount]&+&O(\varepsilon),
\end{eqnarray*}
thus,
\begin{equation}
\begin{array}{ll}
Q_{\lambda}(u_{a,\varepsilon})&=p_{_{0}}
S+\frac{\omega_{_{3}}\int_{0}^{R}\zeta^{2}(r)dr}{K_{2}}\left[\frac{\int_{0}^{R}(p_{_{0}}+\beta_{k}r^{k})|\zeta'(r)|^{2}dr+k
\int_{0}^{R}|\zeta(r)|^{2}r^{k-2}dr}{\int_{0}^{R}|\zeta(r)|^{2}dr}-\lambda\right]\varepsilon^{\frac{1}{2}}\\[\medskipamount]&+O(\varepsilon).
\end{array}
\label{eq24}
\end{equation}

$$\textrm{Set\quad}D(k,\zeta)=\frac{\int_{0}^{R}(p_{_{0}}+\beta_{k}r^{k})|\zeta'(r)|^{2}dr+k
\int_{0}^{R}|\zeta(r)|^{2}r^{k-2}dr}{\int_{0}^{R}|\zeta(r)|^{2}dr}
\textrm{\,\, and\,\,} \gamma(k)=\inf_{H}D(k,\zeta)$$where $H$ is
defined by\\~\\$H=\{\zeta\in C_{0}^{\infty}(\bar{\Omega}),\,
0\leq\zeta\leq1,\,$ $\zeta=1$ in $\{x,\, |x-a|<\frac{R}{2}\}$ and
$\zeta=0$ in $\{x,\, |x-a|\geq R\}\}$.\\This finishes the proof of
Lemma {\ref{lmhy2}}.
\end{proof} Now, we go back to proof of assertion 3) in Theorem~{\ref{thhy2}}
(case $0<k<2$).\\First of all, let us emphasize that if the domain
$\Omega$ is starshaped about $a$, the assertion 3) is more
interesting. Indeed, it gives a better estimate of the least value
of the parameter $\lambda$ over which there is a solution to problem
(\ref{eqhy1}).\\
In the case of a non-starshaped domain, combining the fact that
$S_{0}(p)=p_{_{0}}S$ with the properties of $S_{\lambda}(p)$ (see
the proof of lemma~{\ref{moi}}), we have that there exists
$\lambda^{*} \in [0,\lambda_{1}^{div}[$ such that for all
$\lambda\in ]\lambda^{*},\lambda_{1}^{div}[$, the problem
(\ref{eqhy1}) has a solution. Note that we have no other information
on $\lambda^{*}$.\\Therefore, throughout the rest of this proof, we
assume that the domain $\Omega$ is starshaped about $a$.\\We need
two Lemmas. Let us start by the following
\begin{lemma}~\\
Assume $0<k\leq2$. Then there exists a constant
$\tilde{\beta_{k}}=\beta_{k}\min[(\diam \Omega)^{k-2},1]$ such that
\ba S_{\lambda}(p)=p_{_{0}} S \textrm{\, for every\, $\lambda \in
]-\infty, \tilde{\beta}_{k} \frac{n^{2}}{4}]$} \label{eqhy17} \ea
and the infimum of $S_{\lambda}(p)$ is not achieved for every
$\lambda \in ]-\infty, \tilde{\beta}_{k} \frac{n^{2}}{4}[$.
\label{lmk}
\end{lemma}
\begin{proof} We know from~(\ref{eqhy16}) that
$$
S_{\lambda}(p)\leq Q_{\lambda}(u_{a,\varepsilon})\leq p_{_{0}} S+
A_{_{k}}
 \varepsilon^{\frac{k}{2}} + o(\varepsilon^{\frac{k}{2}})\textrm{\quad with $A_{_{k}}$ is a positive constant},
$$
thus
$$
S_{\lambda}(p)\leq p_{_{0}} S.
$$
On the other hand, we know from Lemma~\ref{poho} and
Proposition~\ref{pr1}, that for $0<k\leq2$, for every $\lambda\leq
\frac{k}{2}\beta_{k}(\frac{n+k-2}{2})^{2}\left(\diam
\Omega\right)^{k-2}$, problem (\ref{eqhy1}) has no solution. So we
exclude the case $S_{\lambda}(p)<p_{_{0}}S$, otherwise,
Lemma~\ref{lmhy3} will yield in a contradiction.\\We conclude that
for $0<k\leq 2$, we have \ba S_{\lambda}(p)=p_{_{0}}S \textrm{\quad
for every\quad $\lambda \leq
\frac{k}{2}\beta_{k}(\frac{n+k-2}{2})^{2}$}\left( \diam
\Omega\right)^{k-2}. \label{eqhy18} \ea Now, we consider $\tilde{p}
$ defined by \ba\left\{
\begin{array}{lll}
\tilde{p}(x)&=p(x) \qquad &\forall x\in \Omega\setminus B(a,r),\\
\tilde{p}(x)&=p_{_{0}}+\beta_{k}|x-a|^{2}\qquad&\forall x\in
B(a,\frac{r}{2}),\\p(x)&\ge\tilde{p}(x)\qquad &\forall x\in
B(a,r)\setminus B(a,\frac{r}{2}),\label{p'}
\end{array}
\right. \ea where $r<1$ is a positive constant.\\Since $0<k\leq2$,
we have $|x-a|^{k}\ge|x-a|^{2}$ for every $x\in B(a,r)$ and $p(x)\ge
\tilde{p}(x)$ in $\Omega$.\\Let $u\in H_{0}^{1}(\Omega)$ with
$\|u\|_{q}=1$, then
\begin{equation*}
\int_{\Omega}p(x)|\nabla u(x)|^{2}dx-\lambda
\int_{\Omega}|u(x)|^{2}dx\geq\int_{\Omega}\tilde{p}(x)|\nabla
u(x)|^{2}dx-\lambda \int_{\Omega}|u(x)|^{2}dx,\end{equation*} thus,
\begin{equation}
\begin{array}{lll}
\displaystyle\int_{\Omega}\hspace{-1mm}p(x)|\nabla
u(x)|^{2}dx-\lambda
\int_{\Omega}|u(x)|^{2}dx\hspace{-2mm}&\geq\hspace{-1mm}\displaystyle\int_{\Omega}
(p_{_{0}}+\frac{1}{2}(\tilde{p}(x)-p_{_{0}}))|\nabla
u(x)|^{2}dx\\[\medskipamount]
\hspace{-2mm}&-\lambda\displaystyle\int_{\Omega}\hspace{-1mm}|u(x)|^{2}dx+\frac{1}{2}\int_{\Omega}\hspace{-1mm}(\tilde{p}(x)-p_{_{0}})|\nabla
u(x)|^{2}dx.
\end{array}
\label{eqk}
\end{equation}
Set
$\tilde{\tilde{p}}(x)=p_{_{0}}+\frac{1}{2}(\tilde{p}(x)-p_{_{0}})$.\\
From (\ref{eqhy00}) we deduce that \begin{equation}
p(x)-p_{_{0}}\geq \beta_{k}|x-a|^{k}\textrm{\, a.e in
$\Omega$}.\label{p}\end{equation} Using (\ref{p'}) and (\ref{p}), a
simple computation gives
$\tilde{p}(x)-p_{_{0}}\ge\tilde{\beta}_{k}|x-a|^{2}$\, a.e in
$\Omega$, with $\tilde{\beta_{k}}=\beta_{k}\min[(\diam
\Omega)^{k-2},1]$.\\Applying Lemma~{\ref{lmhy1}}, we find
$$
\int_{\Omega}(\tilde{p}(x)-p_{_{0}})|\nabla
u(x)|^{2}dx\ge\tilde{\beta}_{k}\frac{n^{2}}{4}\int_{\Omega}|u(x)|^{2}dx.
$$
Inequality (\ref{eqk}) becomes for every $u\in H_{0}^{1}(\Omega)$,
\basn \int_{\Omega}p(x)|\nabla u|^{2}dx-\lambda
\int_{\Omega}|u|^{2}dx\geq \int_{\Omega}\tilde{\tilde{p}}(x)|\nabla
u|^{2}dx-\left(\lambda-\tilde{\beta}_{k}\frac{n^{2}}{8}\right)
\int_{\Omega}|u|^{2}dx. \easn Thus, we find \basn S_{\lambda}(p)\geq
\inf_{\|u\|_{q}^{2}=1}\left[\int_{\Omega}
\tilde{\tilde{p}}(x)|\nabla
u|^{2}dx-(\lambda-\tilde{\beta}_{k}\frac{n^{2}}{8})
\int_{\Omega}|u|^{2}dx\right]. \easn On the other hand
$\lambda-\tilde{\beta}_{k}\frac{n^{2}}{8}\leq\frac{1}{2}\tilde{\beta}_{k}\frac{n^{2}}{4}$
since $\lambda\leq \tilde{\beta}_{k}\frac{n^{2}}{4}$, so by
(\ref{eqhy18}), we conclude that
$$
\inf_{\|u\|_{q}=1}\left[\int_{\Omega} \tilde{\tilde{p}}(x)|\nabla
u|^{2}dx-(\lambda-\tilde{\beta}_{k}\frac{n^{2}}{8})
\int_{\Omega}|u|^{2}dx\right]=p_{_{0}}S,
$$
hence, (\ref{eqhy17}) follows.\\
Now, we are able to prove that the infimum in (\ref{eqhy17}) is not
achieved. Suppose by contradiction that it is achieved by some
$u_{0}$. Let $\delta$ such that
$\tilde{\beta}_{k}\frac{n^{2}}{4}\ge\delta > \lambda$. Using $u_{0}$
as a test function for $S_{\delta}$, we obtain \basn
\begin{split}
 S_{\delta}(p) \leq\frac{\int_{\Omega}p(x)|\nabla u_{0}|^{2}dx-\delta
\int_{\Omega}|u_{0}|^{2}dx}{\|u_{0}\|_{q}^{2}}
<\frac{\int_{\Omega}p(x)|\nabla u_{0}|^{2}dx-\lambda
\int_{\Omega}|u_{0}|^{2}dx}{\|u_{0}\|_{q}^{2}}
\end{split}
 \easn
and thus $S_{\delta}(p)<S_{\lambda}(p)=p_{_{0}}S$. This is a
contradiction since $S_{\delta}(p)=p_{_{0}}S$ for $\delta \leq
\tilde{\beta}_{k}\frac{n^{2}}{4}$.
\end{proof}
The second Lemma on which the proof of assertion 3) in
Theorem~{\ref{thhy2}} is based is the following
\begin{lemma}~\\
There exists $\lambda^{*}\in [\tilde{\beta}_{k}\frac{n^{2}}{4},
\lambda_{1}^{\textsl{div}}[$, such that for all
$\lambda\in]\lambda^{*},\lambda_{1}^{div}[$ we
have$$S_{\lambda}(p)<p_{_{0}}S.$$\label{moi}
\end{lemma}
\begin{proof}~\\The proof is based on a study of some properties of the function
$\lambda\mapsto S_{\lambda}(p)$. We have
$S_{\lambda_{1}^{\textsl{div}}}(p)=0$. Indeed let $\varphi_{1}$ be
the eigenfunction of $\textsl{div}(p\nabla .)$ corresponding to
$\lambda_{1}^{div}$, we have \basn
\begin{split}
S_{\lambda_{1}^{\textsl{div}}}\leq \frac{\int p(x)|\nabla
\varphi_{1}|^{2}dx-\lambda_{1}^{div}\int |\varphi_{1}|^{2}dx}{(\int
|\varphi_{1}|^{q}dx)^{\frac{2}{q}}}=0.
\end{split}
\easn Moreover, $\lambda\mapsto S_{\lambda}(p)$ is continuous and
$S_{\tilde{\beta}_{k}\frac{n^{2}}{4}}(p)=p_{_{0}}S $. Then according
to the Mean Value Theorem, there exists
$\beta\in]\tilde{\beta}_{k}\frac{n^{2}}{4},\lambda_{1}^{\textsl{div}}[$
such that $0<S_{\beta}(p)<p_{_{0}}S$. But the function
$\lambda\mapsto S_{\lambda}(p)$ is decreasing hence $\forall
\lambda\in [\beta, \lambda_{1}^{\textsl{div}}[$ we have
$S_{\lambda}(p)<p_{_{0}}S$, and the Lemma follows at once.
\end{proof}
Now we have all the necessary ingredients for the proof of
Theorem~{\ref{thhy2}}.\\{\bf Proof of Theorem~{\ref{thhy2}}
concluded}: Concerning the proof of 1), 2), 3) and 4), let $u\in
H_{0}^{1}(\Omega)$ be given by Lemma~\ref{lmhy3}, that is,
$$\|u\|_{q}=1\,\,\, \textrm{and\,\, $\int_{\Omega}p(x)|\nabla
u(x)|^{2}dx-\lambda \int_{\Omega}|u(x)|^{2}dx=S_{\lambda}(p)$.}
$$
We may as well assume that $u\geq0$. Since $u$ is a minimizer for
(\ref{eqhy8}) there exists a Lagrange multiplier $\mu\in \R$ such
that
$$
-div(p\nabla u)-\lambda u=\mu u^{q-1}  \textrm{ on $\Omega$.}
$$
In fact, $\mu=S_{\lambda}(p)$, and $S_{\lambda}(p)>0$ since
$\lambda<\lambda_{1}^{\textsl{div}}$. It follows that $\gamma u$
satisfies (\ref{eqhy1}) for some appropriate constant $\gamma>0$
($\gamma=(S_{\lambda}(p))^{\frac{1}{q-2}}$), note that $u>0$ on
$\Omega$ by the strong maximum principle. \\
Now we prove the assertion 5) of Theorem~\ref{thhy2}. From
(\ref{eqhy16}) and since $\lambda \leq0$ we have
$$
p_{_{0}}S\leq S_{\lambda}(p)\leq Q_{\lambda}(u_{a, \varepsilon})\leq
p_{_{0}}S+o(1).$$ Hence $S_{\lambda}(p) =p_{_{0}}S$ and the infimum
is not achieved, indeed we suppose that $S_{\lambda}(p)$ is achieved
by some function $u\in H_{0}^{1}(\Omega)$, in that case
$$
S_{\lambda}(p)=\int_{\Omega}p(x)|\nabla u(x)|^{2}dx-\lambda
\int_{\Omega}|u(x)|^{2}dx,  \textrm{    with $\|u\|_{q}=1$.}
$$
Using the fact that $S$ is not attained and since $\lambda\leq 0$,
we deduce
$$
p_{_{0}}S<p_{_{0}}\int_{\Omega}|\nabla u(x)|^{2}dx\leq
S_{\lambda}(p)=p_{_{0}}S,
$$
then we obtain a contradiction.\\Finally we prove assertion 6) in
Theorem~{\ref{thhy2}}. Let $\varphi{_{1}}$ be the eigenfunction
corresponding to $\lambda_{1}^{div}$ with $\varphi{_{1}}>0$ on $
\Omega$. Suppose that $u$ is a solution of~(\ref{eqhy1}). We have
\begin{equation*}
\begin{array}{lll}
\displaystyle-\int_{\Omega} \textsl{div}(p(x)\nabla u(x))
\varphi{_{1}}(x)dx&=\displaystyle\lambda^{\textsl{div}}_{1}
\int_{\Omega} u(x)
\varphi_{1}(x) dx\\[\bigskipamount]
&=\displaystyle\int_{\Omega} u^{q-1}(x) \varphi{_{1}}(x)dx+ \lambda
\int_{\Omega}u(x) \varphi{_{1}}(x)dx,\end{array}\end{equation*} thus
$$\lambda^{\textsl{div}}_{1} \int_{\Omega} u(x) \varphi{_{1}}(x) dx >
\lambda \int_{\Omega}u(x) \varphi_{_{1}}(x)dx$$ and
$$\lambda^{\textsl{div}}_{1}>\lambda.$$This completes the proof of
Theorem~{\ref{thhy2}}.
%%%%%%%%%%%%%%%%%%%
%%%%%%%%%%%%%%%
%%%%%%%%%%%%
\section{The effect of the geometry of the domain}
Let $\Omega\subset\R^{n}$, $n\ge 3$, be a bounded domain. We study the equation\\
\ba \left\{
\begin{array}{lll}
-\textsl{div}(p(x)\nabla u)=u^{q-1} &\textrm{in $\Omega$,}\\
 \hspace{23.5mm} u>0 &\textrm{in $\Omega$,}\\
 \hspace{23mm}u=0 &\textrm{on $\partial\Omega$.}
\end{array}
\right. \label{equation26} \ea where $q=\frac{2n}{n-2}$ and $ p:
\bar{\Omega}\longrightarrow \R$ is a positive weight belonging to
$C(\bar{\Omega})\cap H_{0}^{1}(\Omega)$.\\We assume in this section
that $p$ is such that $\nabla p(x).(x-a)\geq {0}$ a.e $x\in\Omega$ and we set $p_{_{0}}=p(a)$.\\
Let us start by the following non-existence result
\begin{lemma}~\\
\label{lmhy5} There is no solution of~(\ref{equation26}) if
\,$\,\Omega$ is a starshaped domain about $a$.
\end{lemma}
\medskip
\begin{proof} This follows from Pohozaev's identity.\\
Suppose that $u$ is a solution of (\ref{equation26}), we have (see
Lemma~{\ref{poho}} Section 2 for $\lambda=0$),
\begin{center}
\ba \int_{\Omega}\nabla p(x) .(x-a)\vert\nabla
u(x)\vert^{2}dx+\int_{\partial\Omega}p(x)[(x-a).\nu]\vert\frac{\partial
u}{\partial \nu}\vert^{2}dx=0. \label{equation27} \ea
\end{center}
Note that $(x-a).\nu>{0}$ a.e
on $\partial\Omega$ since $\Omega$ is starshaped about $a$.\\
Since $\nabla p(x).(x-a)\geq {0}$ a.e $x\in\Omega$, we deduce
from~(\ref{equation27}) that $\frac{\partial u}{\partial \nu}=0$ on
$\partial{\Omega}$, and then by~(\ref{equation26}) we have
$$\int_{\Omega}u^{q-1}(x)dx=-\int_{\Omega}div(p(x)\nabla u(x))dx=\int_{\partial\Omega}\frac{\partial u}{\partial \nu} dx=0 ,$$
thus $$u\equiv 0 .$$
\end{proof}
Suppose that $\Omega$ is starshaped about $a$. In view of
Lemma~\ref{lmhy5}, we will modify the geometry of $\Omega$ in order
to find a solution of problem~\ref{equation26}. For a $\varepsilon
>0 $ small enough, we set
$\Omega_{\varepsilon}=\Omega\setminus \bar{B}(a,\varepsilon)$.\\
We investigate the problem (\ref{equation26}) in the new domain
$\Omega_{\varepsilon}$, and, throughout the rest of this paper, we
shall denote this new problem by $(I_{\varepsilon})$
.\\
Since $p$ is a continuous function, then $\forall$ $\theta>0$,
$\exists\, r_{0}>0$ such that $\forall\, \sigma\in\Sigma$, where
$\Sigma$ designates the unit sphere of $\R^{n}$, we have
$|p(a+r_{0}\sigma)-p_{_{0}}|<\frac{\theta}{2S^{\frac{n}{2}}}$.\\Throughout
the rest of this Section, $\theta>0$ is fixed, small enough, and
$r_{0}>0$ is given as the previous definition.\\We recall the main
result of this section which we have already stated by
theorem~\ref{thhy0} in the introduction
\begin{theorem}~\\
\label{thhy3}There exists
$\varepsilon_{0}=\varepsilon_{0}(\Omega,p)\leq r_{0}$ such that for
every $0<\varepsilon<\varepsilon_{0}$, the problem
$(I_{\varepsilon})$ has at least one solution in
$H^{1}_{0}(\Omega_{\varepsilon})$.
\end{theorem}
In order to prove the Theorem~\ref{thhy3}, we need to apply the
following result, see \cite{ar},
\begin{thoa}~\\
Let $E$ be a $C^{1}$ function defined on a Banach
space $X$, and let $K$ a compact metric space. We denote by $K^{*}$ a nonempty subset of $K$, closed, different from $ K$ and we fix $f^{*}\in C(K^{*},X)$.\\
We define $\mathcal{P} =\{f\in C(K,X) / f=f^{*} $on $ K^{*}\}$ and $
c=\inf_{f\in\mathcal{P}}\sup_{t\in K} E(f(t))
$\\
Suppose that for every $f$ of $\mathcal{P}$, we have
$$
\max_{t\in K} E(f(t))>\max_{t\in K^{*}}E(f(t)),
$$
then there exists a sequence  $(u_{j})\subset X$ such that
$E(u_{j})\longrightarrow c$ and $E'(u_{j})\longrightarrow$ 0 in
$X^{*}$. %\label{tha}
\end{thoa}~\\
We consider the functional
$$
E(u)=\frac{1}{2}\int_{\Omega_{\varepsilon}}p(x)|\nabla u(x)|^{2}dx
 -\frac{1}{q}\int_{\Omega_{\varepsilon}}|u(x)|^{q}dx.
$$
%\begin{theo}
%{\bf Theorem B}\\
%\label{th5}
In addition to Theorem A 1, the proof of Theorem~\ref{thhy3}
requires the following result (see \cite{b} and Proposition 2.1 in
\cite{s})
\begin{thoa}~\\
Suppose that for some sequence $(u_{j})\subset
H_{0}^{1}(\Omega_{\varepsilon})$ we have $E(u_{j})\rightarrow
c\in]\frac{1}{n}(p_{_{0}}S)^{\frac{n}{2}},\frac{2}{n}(p_{_{0}}S)^{\frac{n}{2}}[$
and $dE(u_{j})\rightarrow 0$ in $H^{-1}(\Omega_{\varepsilon})$. Then
$(u_{j})$ contains a strongly convergent subsequence. %\label{thc}
\end{thoa}
Now, we return to the proof of Theorem~\ref{thhy3}.\\We shall need
the following functions: \basn
&\Gamma&:H^{1}_{0}(\Omega_{\varepsilon})\longrightarrow \R,\quad
\Gamma(u)=\int_{\Omega_{\varepsilon}}p(x)\vert\nabla
u(x)\vert^{2}dx-\int_{\Omega_{\varepsilon}}|u(x)|^{q}dx.\\
&F&:H^{1}_{0}(\Omega_{\varepsilon})\longrightarrow \R^{n},\quad
F(u)=(p_{_{0}}S)^{-\frac{n}{2}}\int_{\Omega_{\varepsilon}}x p(x)|\nabla u(x)|^{2}dx.\\
\easn We have the following result
\begin{lemma}~\\
For every neighborhood $V$ of $\bar{\Omega}_{\varepsilon}$ there
exists $\eta >0$ such that if $u\not=0$, $\Gamma(u)=0$ and
$E(u)\leq\frac{1}{n}(p_{_{0}}S)^{\frac{n}{2}}+2\eta$, then $F(u)\in
V$. \label{lmhy6}
\end{lemma}
\begin{proof} We proceed by contradiction. We assume that there exists $V$ a
compact neighborhood of $\bar{\Omega}_{\varepsilon}$ not containing
$a$, such that $\forall j\in \N^{*}$, we have \basn
u_{j}&\not=&0,\\
\Gamma(u_{j})&=&0,\\
E(u_{j})&\le&\frac{1}{n}(p_{_{0}}S)^{\frac{n}{2}}+\frac{1}{j},\\
F(u_{j})&\not\in& V. \easn Since $\Gamma(u_{j})=0$, we see that
\basn \hspace{-23mm}\int_{\Omega_{\varepsilon}}p(x)|\nabla
u_{j}|^{2}dx&=&\int_{\Omega_{\varepsilon}}|u_{j}|^{q}dx\easn and
\basn\int_{\Omega_{\varepsilon}}p(x)|\nabla
u_{j}|^{2}dx&=&\left(\frac{\int_{\Omega_{\varepsilon}}p(x)|\nabla
u_{j}|^{2}dx}{\left(\int_{\Omega_{\varepsilon}}|u_{j}|^{q}dx\right)^{\frac{2}{q}}}\right)^{\frac{n}{2}}.
\easn Consequently
\begin{equation*}
E(u_{j})=\frac{1}{n}\int_{\Omega_{\varepsilon}}p(x)|\nabla
u_{j}(x)|^{2}dx.\end{equation*} Using the definition of $u_{j}$, the
fact that $p_{_{0}}=\min_{\bar{\Omega}} p(x)$ and the definition of
$S$, we write \basn
\frac{1}{n}(p_{_{0}}S)^{\frac{n}{2}}\le\frac{1}{n}\left(\frac{p_{_{0}}\int_{\Omega_{\varepsilon}}|\nabla
u_{j}|^{2}dx}{\left(\int_{\Omega_{\varepsilon}}|u_{j}|^{q}dx\right)^{\frac{2}{q}}}\right)^{\frac{n}{2}}
\le E(u_{j})\le \frac{1}{n}(p_{_{0}}S)^{\frac{n}{2}}+\frac{1}{j}
\easn and we deduce
\begin{equation*}
\int_{\Omega_{\varepsilon}}p(x)|\nabla
u_{j}(x)|^{2}dx=(p_{_{0}}S)^{\frac{n}{2}}+o(1).
\end{equation*}Applying the Theorem 2 in \cite{c}, (see also Lemma I.1 and Lemma I.4 in \cite{l}), for a subsequence of $(u_{j})_j$ still denoted by
$(u_{j})_j$, there exists $x_{0}\in\bar{\Omega}_{\varepsilon}$ such
that
$$
p(x)|\nabla u_{j}|^{2}\longrightarrow
(p_{_{0}}S)^{\frac{n}{2}}\delta_{x_{0}} \,(j\rightarrow \infty),
$$
where the above convergence is understood for the weak topology of
bounded measures on $\bar{\Omega}_{\varepsilon}$ and where
$\delta_{x_{0}}$ is the Dirac measure at $x_{0}$.\\As a consequence,
$F(u_{j})\in\bar{\Omega}_{\varepsilon}\subset V$, and this
contradicts the hypothesis.
\end{proof}Let $R_{0}>0$ such that $B(a,2R_{0})\subset \Omega$.\\For $k\in
\N^{*}$, let $\varphi_{k}\in C^{\infty}(\R^{n},[0,1])$ such that
$$ \left\{
\begin{array}{ll}
\varphi_{k}(x)=0 &\textrm{if $\vert x-a\vert\le{\frac{1}{4k^{2}}}$
and if
$| x-a|\ge2R_{0}$},\\
\varphi_{k}(x)=1 &\textrm{if $\frac{1}{2k^{2}}\le\vert x-a\vert\le
{R_{0}}$.}
\end{array}
\right.
$$
We consider the family of functions
$$
u^{\sigma}_{t}(x)=\left[\frac{1-t}{(1-t)^{2}+\vert
x-a-t\sigma\vert^{2}}\right]^{\frac{n-2}{2}},
$$
where $t\in [0,1[$, $\sigma\in\Sigma$ and where $\Sigma$ denotes
the unit sphere of $\R^{n}$.\\
We see easily that $\int_{\R^{n}}\vert\nabla
u^{\sigma}_{t}\vert^{2}dx$ and $\int_{\R^{n}}\vert
u^{\sigma}_{t}\vert^{q}dx$ are independent of $t\in [0,1[$ and of
$\sigma\in\Sigma$. We also have
$$
\int_{\R^{n}}\vert\nabla
u^{\sigma}_{t}(x)\vert^{2}dx=S\left(\int_{\R^{n}}\vert
u^{\sigma}_{t}(x)\vert^{q}dx\right)^{\frac{2}{q}}. \label{eqt}
$$
We set
$$v^{\sigma}_{t,k}(x)=\frac{(1-t)^{\frac{n-2}{2}}k^{\frac{n-2}{2}}\varphi_{k}(x)}{((1-t)^{2}+\vert
k(x-a-tr_{0}\sigma)\vert^{2})^{\frac{n-2}{2}}},$$ we remark that $v^{\sigma}_{t,k}\in H^{1}_{0}(\Omega_{\varepsilon})$. For $r>0$, let $g(r)=E(rv^{\sigma}_{t,k})$, then\\
$rg'(r)=\Gamma(rv^{\sigma}_{t,k})$, $g(r)\rightarrow -\infty$,
when $r\rightarrow +\infty$, $g(0)=0$ and $g(r)>0$ for $r>0$ small enough.\\
We conclude, from the above, that $g$ reaches its maximum at
$$r=\left[\frac{\int_{\Omega_{\varepsilon}}p(x)|\nabla v_{t,k}^{\sigma}|^{2}dx}{\int_{\Omega_{\varepsilon}}|v_{t,k}^{\sigma}|^{q}dx}\right]^{\frac{1}{q-2}}>0.$$
We set $w^{\sigma}_{t,k}=r v^{\sigma}_{t,k}$. We have
\begin{lemma}~\\The following two statements are true:
\begin{equation*}\begin{array}{lllll}\textrm{\bf a)}\forall \delta>0,\, \exists
k_{0}\ge 1\,\,\mbox{such that}\,\,( \forall \,k\ge
k_{0})\,\textrm{then}&\\&\hspace{-30mm}(\forall\sigma\in\Sigma\,\,\textrm{and\,}\,\forall t\in[0,1[,\,\, E(w^{\sigma}_{t,k})\le\frac{1}{n}(p_{_{0}}S)^{\frac{n}{2}}+\delta)\\
\textrm{\bf b)}\forall \alpha>0,\, \exists\mu>0\,\,\mbox{such
that}\,\,(\mu<t<1)\,\textrm{then}&\\&\hspace{-30mm}(\forall\sigma\in\Sigma\,\,\textrm{and\,}\,\forall
k\ge
1,\,\,E(w^{\sigma}_{t,k})\hspace{-1mm}\le\frac{1}{n}(p_{_{0}}S)^{\frac{n}{2}}+\alpha)
&\\&\hspace{-30mm}\mbox{and\,}\quad\vert
F(w^{\sigma}_{t,k})\hspace{-1mm}-\hspace{-1mm}(a+r_{0}\sigma)\vert\le\alpha.
\end{array}\end{equation*} \label{lm7}
\end{lemma}
\begin{proof} Before proving this Lemma, let us remark that the function
$v_{t,k}^{\sigma}$ corresponds to the function $u_{a,\varepsilon}$
defined in the beginning of this paper, so for more details of
calculus we refer to section 2.\\We start by proving the assertion
a). Let $t\in[0,1[$, we have \basn
E(w^{\sigma}_{t,k})&=&\frac{1}{2}\int_{\Omega_{\varepsilon}}p(x)|\nabla
w^{\sigma}_{t,k}|^{2}dx-\frac{1}{q}\int_{\Omega_{\varepsilon}}\vert
w^{\sigma}_{t,k}\vert^{q}dx,\\[\medskipamount]
&=&\frac{r^{2}}{2}\int_{\Omega_{\varepsilon}}p(x)|\nabla
v^{\sigma}_{t,k}|^{2}dx-\frac{r^{q}}{q}\int_{\Omega_{\varepsilon}}\vert
v^{\sigma}_{t,k}\vert^{q}dx.\easn Using the definition of $r$, the
definition of $\varphi_{k}$ and applying the Dominated Convergence
Theorem, we obtain, as $k\rightarrow\infty$,\basn
E(w^{\sigma}_{t,k})&=&\frac{1}{n}\left[\frac{k^{n}(n-2)^{2}(1-t)^{n-2}\int_{_{\{\frac{1}{2k^{2}}\leq|x-a|\leq{R_{0}}\}}}p(x)\frac{|k
(x-a-tr_{0}\sigma)|^{2}}{((1-t)^{2}+|k
(x-a-tr_{0}\sigma)|^{2})^{n}}dx}{\left[k^{n}(n-2)^{2}(1-t)^{n}\int_{_{\{\frac{1}{2k^{2}}\leq|x-a|\leq{R_{0}}\}}}\frac{1}{((1-t)^{2}+|k
(x-a-tr_{0}\sigma)|^{2})^{n}}dx\right]^{\frac{2}{q}}}\right]^{\frac{
n}{2}}\\[\medskipamount]
&+&o(1).\easn By the following change of variable $y=\frac{k
(x-a-tr_{0}\sigma)}{1-t}$, we see that\basn
E(w^{\sigma}_{t,k})&=&\frac{1}{n}\left[\displaystyle\frac{(n-2)^{2}\int_{_{\{\frac{1}{2k(1-t)}-\frac{tr_{0}}{1-t}\leq|y|\leq{\frac{kR_{0}}{1-t}+\frac{tr_{0}}{1-t}}\}}}p(\frac{y(1-t)}{k}+a+tr_{0}\sigma)\frac{|y|^{2}}{(1+|y|^{2})^{n}}dy}{\left[\int_{_{\{\frac{1}{2k(1-t)}-\frac{tr_{0}}{1-t}\leq|y|\leq{\frac{kR_{0}}{1-t}+\frac{tr_{0}}{1-t}}\}}}\hspace{-2mm}\frac{1}{(1+|y|^{2})^{n}}dy\right]^{\frac{2}{q}}}\hspace{-1mm}\right]^{\frac{n}{2}}\\[\medskipamount]
&+&o(1).\easn Applying again the Dominated Convergence Theorem, we
deduce, as $k\rightarrow\infty$, that \basn
E(w^{\sigma}_{t,k})&=&\frac{1}{n}\left[\frac{(n-2)^{2}p(a+tr_{0}\sigma)\int_{\R^{n}}\frac{|y|^{2}}{(1+|y|^{2})^{n}}dy}{\left[\int_{\R^{n}}\frac{1}{((1-t)^{2}+|y|^{2})^{n}}dy\right]^{\frac{2}{q}}}\right]^{\frac{n}{2}}+o(1),\\[\medskipamount]
&=&\frac{1}{n}(p(a+tr_{0}\sigma))^{\frac{n}{2}}S^{\frac{n}{2}}+o(1).\easn
Now, using the definition of $r_{0}$, a simple computation shows
that $\forall \delta>0$, $\exists k_{0}\ge 1$ such that $\forall
k\ge k_{0}$, we have
$$E(w^{\sigma}_{t,k})\leq\frac{1}{n}(p_{_{0}}S)^{\frac{n}{2}}+\delta,
$$ which finishes the proof of a).\\Now we return to the proof of
b), let $k\in\N^{*}$, we have\basn
E(w^{\sigma}_{t,k})&=&\frac{1}{2}\int_{\Omega_{\varepsilon}}p(x)|\nabla
w^{\sigma}_{t,k}|^{2}dx-\frac{1}{q}\int_{\Omega_{\varepsilon}}\vert
w^{\sigma}_{t,k}\vert^{q}dx\\[\medskipamount]
&=&\frac{r^{2}}{2}\int_{\Omega_{\varepsilon}}p(x)|\nabla
v^{\sigma}_{t,k}|^{2}dx-\frac{r^{q}}{q}\int_{\Omega_{\varepsilon}}\vert
v^{\sigma}_{t,k}\vert^{q}dx. \easn Looking at the definition of
$\varphi_{k}$ and $r$, we easily see, as $t\rightarrow 1$, that
\basn
E(w^{\sigma}_{t,k})=\frac{1}{n}\left[\frac{k^{n}(n-2)^{2}(1-t)^{n-2}\int_{\R^{n}}p(x)\frac{|k
(x-a-tr_{0}\sigma)|^{2}}{((1-t)^{2}+|k
(x-a-tr_{0}\sigma)|^{2})^{n}}dx}{\left[k^{n}(1-t)^{n}\int_{\R^{n}}\frac{1}{((1-t)^{2}+|k
(x-a-tr_{0}\sigma)|^{2})^{n}}dx\right]^{\frac{2}{q}}}\right]^{\frac{n}{2}}+O((1-t)^{n-2}).\easn
 By the change of variable $y=\frac{k(x-a-tr_{0}\sigma)}{1-t}$, we get
\basn
E(w^{\sigma}_{t,k})=\frac{1}{n}\left[\frac{(n-2)^{2}\int_{\R^{n}}p(\frac{(1-t)y}{k}+a+tr_{0}\sigma)\frac{|y|^{2}}{(1+|y|^{2})^{n}}dy}{\left[\int_{\R^{n}}\frac{1}{(1+|y|^{2})^{n}}dy\right]^{\frac{2}{q}}}\right]^{\frac{n}{2}}+O((1-t)^{n-2}).\easn
Applying the Dominated Convergence Theorem, we obtain \basn
E(w^{\sigma}_{t,k})&=&\frac{1}{n}\left[\frac{(n-2)^{2}p(a+r_{0}\sigma)\int_{\R^{n}}\frac{|y|^{2}}{(1+|y|^{2})^{n}}dy}{\left[\int_{\R^{n}}\frac{1}{(1+|y|^{2})^{n}}dy\right]^{\frac{2}{q}}}\right]^{\frac{n}{2}}+O((1-t)^{n-2}),\\[\medskipamount]
&=&\frac{1}{n}(p(a+r_{0}\sigma))^{\frac{n}{2}}S^{\frac{n}{2}}+O((1-t)^{n-2}).\easn
Using the definition of $r_{0}$, a simple computation shows that
$\forall \alpha>0,$ $\exists \mu>0$ such that $\forall\,\mu<t<1$, we
have $$
E(w^{\sigma}_{t,k})\leq\frac{1}{n}(p_{_{0}}S)^{\frac{n}{2}}+\alpha.$$On
the other hand \basn
F(w^{\sigma}_{t,k})&=&(p_{_{0}}S)^{-\frac{n}{2}}\int_{\R^{n}}x
p(x)|\nabla
w^{\sigma}_{t,k}(x)|^{2}dx,\\[\medskipamount]
&=&(p_{_{0}}S)^{-\frac{n}{2}}r^{2}\int_{\R^{n}}x p(x)|\nabla
v^{\sigma}_{t,k}(x)|^{2}dx.\easn By the definition of
$v_{t,k}^{\sigma}$ and $r$, we write \basn
F(w^{\sigma}_{t,k})=(p_{_{0}}S)^{-\frac{n}{2}}\left[\frac{(1-t)^{n-2}(n-2)^{2}\int_{\R^{n}}p(x)\frac{|k
(x-a-tr_{0}\sigma)|^{2}}{((1-t)^{2}+|k
(x-a-tr_{0}\sigma)|^{2})^{n}}dx}{(1-t)^{n}\int_{\R^{n}}\frac{1}{((1-t)^{2}+|k
(x-a-tr_{0}\sigma)|^{2})^{n}}dx}\right]^{\frac{
2}{q-2}}\times\\[\medskipamount]
\hspace{2mm}(1-t)^{n-2}k^{n}(n-2)^{2}\int_{\R^{n}}x\, p(x)\frac{|k
(x-a-tr_{0}\sigma)|^{2}}{((1-t)^{2}+|k
(x-a-tr_{0}\sigma)|^{2})^{n}}dx+o(1-t). \easn The change of variable
$y=\frac{k(x-a-tr_{0}\sigma)}{1-t}$ gives \basn
F(w^{\sigma}_{t,k})&=&(p_{_{0}}S)^{-\frac{n}{2}}\left[\frac{(n-2)^{2}\int_{\R^{n}}p\left(\frac{(1-t)y}{k}+a+tr_{0}\sigma\right)\frac{|y|^{2}}{(1+|y|^{2})^{n}}dx}{\int_{\R^{n}}\frac{1}{(1+|y|^{2})^{n}}dx}\right]^{\frac{
2}{q-2}}\times\\[\medskipamount]
&~&(n-2)^{2}\int_{\R^{n}}\frac{(\frac{(1-t)y}{k}+a+tr_{0}\sigma)\,
p(\frac{(1-t)y}{k}+a+tr_{0}\sigma)\,|y|^{2}}{(1+|y|^{2})^{n}}dx+o(1-t).
\easn Applying the Dominated Convergence Theorem, we deduce that
\basn
F(w^{\sigma}_{t,k})&=&(p_{_{0}}S)^{-\frac{n}{2}}(p(a+r_{0}\sigma))^{\frac{n}{2}}\left[\frac{(n-2)^{2}\int_{\R^{n}}\frac{|y|^{2}}{(1+|y|^{2})^{n}}dy}{\left[\int_{\R^{n}}\frac{1}{(1+|y|^{2})^{n}}dy\right]^{\frac{2}{q}}}\right]^{\frac{q}{q-2}}\hspace{-2mm}(a+r_{0}\sigma)+o(1-t),\\[\medskipamount]
&=&(p_{_{0}}S)^{-\frac{n}{2}}(p(a+r_{0}\sigma))^{\frac{n}{2}}S^{\frac{n}{2}}(a+r_{0}\sigma)+o(1-t).
 \easn
Using the definition of $r_{0}$ we get the desired result.
\end{proof}
{\bf{Consequences}}\\
Let $V$ be a compact neighborhood of $\bar {\Omega}_{\varepsilon}$
not containing $a$. Let $0<\eta<r_{0}$ small enough, which
corresponds to $V$ as in Lemma~\ref{lmhy6}, verifying
$r_{0}\sigma+\xi\not=a$ for
$|\sigma|=1$ and $|a-\xi|\leq\eta$.\\
By Lemma~\ref{lm7}, there exists $k_{0}\ge 1$ such that :\\
\ba E(w^{\sigma}_{t,k_{0}})\leq
\frac{2}{n}(p_{_{0}}S)^{\frac{n}{2}}-\eta,\,
\forall\sigma\in\Sigma,\, \forall t\in[0,1[.\label{eql27} \ea
\begin{remark}~\\
We choose $\varepsilon_{0}=\varepsilon_{0}(\Omega,p)\leq
\frac{1}{4k_{0}^{2}}$ small enough and such that $\forall
0<\varepsilon<\varepsilon_{0}$ we have $\{x\,\,|x-a|\leq
\varepsilon\}\not\subset V$.
\end{remark}We fix $\lambda>1$, large enough such that $ E(\lambda
w^{\sigma}_{t,k_{0}})<0$, $\forall\sigma\in\Sigma$, $\forall t\in
[0,1[$. In order to apply Theorem A 1, we define the sets $K$,
$K^{*}$
and the function $f^{*}$ as\\
$K=[0,1]\times\bar{B}(a,r_{0})$,\\$K^{*}=\partial
K=[0,1]\times\partial\bar{B}(a,r_{0})\cup\{0,1\}\times\bar{B}(a,r_{0})$ and\\$f^{*}:K\rightarrow H^{1}_{0}(\Omega_{\varepsilon}),$\\
$f^{*}(s,tr_{0}\sigma)=\lambda s w^{\sigma}_{t,k_{0}}$.\\
The conclusion of Theorem \ref{thhy3} follows from the next
\begin{lemma}~\\
\label{lm8} We have
$$
\sup_{K}E(f)\ge \frac{1}{n} (p_{_{0}}S)^{\frac{n}{2}}+2\eta ,\,
\forall f\in \mathcal{P}.
$$
\end{lemma}~\\
We postpone the proof of Lemma~{\ref{lm8}} and we complete the proof
of Theorem \ref{thhy3}. From (\ref{eql27}) we have
$$
\max_{r\geq0}E(r v^{\sigma}_{t,k_{0}})=E(w^{\sigma}_{t,k_{0}})\le
\frac{2}{n}(p_{_{0}}S)^{\frac{n}{2}}-\eta\quad\forall
\sigma\in\Sigma,\quad \forall t\in [0,1[.
$$
From assertion b) of Lemma~{\ref{lm7}} there exists $\mu>0$, we fix
$t_{0}\in]\mu,1[$ such that
$$
\max_{r\geq0}E(r
v^{\sigma}_{t_{0},k_{0}})=E(w^{\sigma}_{t_{0},k_{0}})\le
\frac{1}{n}(p_{_{0}}S)^{\frac{n}{2}}+\eta ,\quad \forall
\sigma\in\Sigma .
$$
then
$$
\max_{\partial
K}E(f^{*})\le\frac{1}{n}(p_{_{0}}S)^{\frac{n}{2}}+\eta\quad\mbox{
and}\quad \sup_{K}E(f^{*})<\frac{2}{n}(p_{_{0}}S)^{\frac{n}{2}}.
$$
So, by Lemma~\ref{lm8},
$$
\sup_{K}E(f)\ge\frac{1}{n}(p_{_{0}}S)^{\frac{n}{2}}+2\eta>\frac{1}{n}(p_{_{0}}S)^{\frac{n}{2}}+\eta\ge\sup_{\partial
K}E(f^{*})
$$
and
$$
c=\inf_{f\in\mathcal{P}}\sup_{t\in
K}E(f)\in]\frac{1}{n}(p_{_{0}}S)^{\frac{n}{2}},\frac{2}{n}(p_{_{0}}S)^{\frac{n}{2}}[.
$$
Applying Theorem A 1 and Theorem A 2, we obtain the conclusion of
Theorem~\ref{thhy3}.\\~\\\deml{\bf{~\ref{lm8}}}. We argue by
contradiction. Suppose that there exists $f\in
C(K,H^{1}_{0}(\Omega_{\varepsilon}))$ with $f=f^{*}$ on $\partial
K$, and $E(f(s,\xi))\le
\frac{1}{n}(p_{_{0}}S)^{\frac{n}{2}}+2\eta ,\, \forall (s,\xi)\in K$.\\
We consider the function $G:K\longrightarrow \R^{n+1}$, defined by
\basn G(s,\xi)=(s,F(f(s,\xi))). \easn We will prove that \ba
\deg(G,K,(\lambda^{-1},a))=1 \label{equation28} .\ea The map
$H:[0,1]\times K\longrightarrow \R^{n+1}$, defined
by\\
$H(t,s,\xi)=tG(s,\xi)+(1-t)(s,\xi)=(s,t F(f(s,\xi))+(1-t)\xi)$\\
is a homotopy between $G$ and $Id_{K}$, where $Id_{K}$ is the
Identity
application of $K$.\\
To get~(\ref{equation28}), we start by checking that $(\lambda^{-1},a)\not\in H(t,\partial K)$.\\
If not, there exists $(s,\xi)\in\partial K$ such that
$H(t,s,\xi)=(\lambda^{-1},a)$, as a consequence $s=\lambda^{-1}$
and\,
$a=tF(f(\lambda^{-1},\xi))+(1-t)\xi=t(F(w^{\sigma}_{t_{0},k_{0}})-\xi)+\xi$.\\
Since $s=\lambda^{-1}\in]0,1[$, we have $\xi\in\partial\bar{B}(a,r_{0})$. But, since $|F(w_{t_{0},k_{0}}^{\sigma})-(a+r_{0}\sigma)|<\eta$ $\forall \sigma\in\Sigma$ (see Lemma \ref{lm7}), the fact that $t(F(w^{\sigma}_{t_{0},k_{0}})-\xi)+\xi=a$, $\xi\in \partial\bar{B}(a,r_{0})$ leads to a contradiction. Then, we deduce that $(\lambda^{-1},a)\not\in H(t,\partial K)$ and consequently $\forall t\in[0,1]$, $\deg(H(t,.),K,(\lambda^{-1},a))$ is well defined.\\
We consider the following sets:\\
$K^{+}=\{(s,\xi)\in K \,\,\vert\,\, \Gamma(f(s,\xi))>0\}\cup
(0,\xi)$,
$K^{-}=\{(s,\xi)\in K \,\,\vert\,\, \Gamma(f(s,\xi))<0\}$ and $K^{0}=\{(s,\xi)\in K\,\, \vert\,\, \Gamma(f(s,\xi))=0\}$.\\
If $(s,\xi)\in \partial K$ then we have
$f(s,\xi)=f^{*}(s,\xi)=\lambda s w^{\sigma}_{t_{0},k_{0}}$ and\basn
\Gamma(f(s,\xi))&=&(s\lambda)^{2}\int_{\Omega_{\varepsilon}}
p(x)|\nabla
w^{\sigma}_{t_{0},k_{0}}(x)|^{2}dx-(s\lambda)^{q}\int_{\Omega_{\varepsilon}}|
w^{\sigma}_{t_{0},k_{0}}(x)|^{q}dx\\[\medskipamount]
\Gamma(f(s,\xi))&=&[(s\lambda)^{2}-(s\lambda)^{q}]\int_{\Omega_{\varepsilon}}
p(x)|\nabla w^{\sigma}_{t_{0},k_{0}}(x)|^{2}dx. \easn Since
$\int_{\Omega_{\varepsilon}} p(x)|\nabla
w^{\sigma}_{t_{0},k_{0}}(x)|^{2}dx>0$, we see that \ba
\mbox{If}\quad (s,\xi)\in\partial K\quad\mbox{and if}\quad 0\leq s<
\lambda^{-1},\quad\mbox{then}\quad(s,\xi)\in K^{+}
\label{equation29} \ea \ba \mbox{If}\quad(s,\xi)\in\partial
K\quad\mbox{and if}\quad\lambda^{-1}< s\le
1,\quad\mbox{then}\quad(s,\xi)\in K^{-} \label{equation30} \ea \ba
\quad(\lambda^{-1},\xi)\in K^{0},\quad\forall \xi \in \partial
\bar{B}(a,r_{0}). \label{equation31} \ea Let $(s,\xi)\in K^{0}$, we
have $\Gamma(f(s,\xi))=0.$ Moreover, since $E(f(s,\xi))\le
\frac{1}{n}(p_{_{0}}S)^{\frac{n}{2}}+2\eta$, looking at
Lemma~\ref{lmhy6}, we deduce that
$$
F(f(s,\xi))\in V.
$$
Consequently $\forall (s,\xi)\in K^{0}$, $F(f(s,\xi))\not=a$ since
$a\not\in V.$\\
Hence $(\lambda^{-1},a)\not\in G(K^{0})=G(K\setminus (K^{+}\cup
K^{-}))$, then \ba
\deg(G,K^{+},(\lambda^{-1},a))+\deg(G,K^{-},(\lambda^{-1},a))=\deg(G,K,(\lambda^{-1},a)).
\label{equation32} \ea On the other hand, since
$(\lambda^{-1},a)\not\in H(t,\partial K)$ $\forall t\in [0,1]$ we
have
$$\deg(H(1,.),K,(\lambda^{-1},a))=\deg(H(0,.),K,(\lambda^{-1},a)).$$
Using the fact that $H(0,.)= G$, $H(1,.)=Id_{K}$ and
$\deg(Id_{K},K,(\lambda^{-1},a))=1$, we deduce
(\ref{equation28}).\\Now, we will prove that \ba
\deg(G,K^{+},(\lambda^{-1},a))=0 \label{equation34} \ea \ba
\deg(G,K^{-},(\lambda^{-1},a))=0. \label{equation35} \ea Fix
$R>\lambda^{-1}$ and let $y\in \R^{n+1}$ such that $|y|\ge R$ then
$y\not\in G(K)$.\\We define the path $r(t)=(t
R+(1-t)\lambda^{-1},a)$, for
$t\in [0,1]$.\\
We claim that $r(t)\not\in G(\partial K^{+})$ $\forall\, t\in [0,1]$.\\
If not, there exists $(s,\xi)\in \partial K^{+}$ with\,
$(Rt+(1-t)\lambda^{-1},a)=(s,F(f(s,\xi)))$. Hence
$s=tR+(1-t)\lambda^{-1}\ge\lambda^{-1}$ and $a=F(f(s,\xi))$. But
$\forall (s,\xi)\in K^{0}$, we have $F(f(s,\xi))\not=a$, then $
(s,\xi)\not\in K^{0}$. Hence $(s,\xi)\in\partial K\cap K^{+}$,
(\ref{equation29}) implies that $s<\lambda^{-1}$ and this
contradicts the fact that $s\ge \lambda^{-1}$. Thus $r(t)\not\in
G(\partial K^{+})$ $\forall\, t\in[0,1]$. Hence $\deg(G,K^{+},r(t))$
is well defined and is independent of $t$.\\Since $(R,a)\not\in
G(K)$ we obtain
$$
\deg(G,K^{+},(R,a))=0.
$$
Using the fact that $$\deg(G,K^{+},r(t))=\deg(G,K^{+},(R,a))\quad
\forall\,t\in [0,1],$$we deduce (\ref{equation34}).\\Similarly, we
prove~(\ref{equation35}) by using the path
$q(t)=(-tR+(1-t)\lambda^{-1},a)$, $t\in [0,1]$. We have that
$\deg(G,K^{-},q(t))$ is independent of $t$. Using the fact that
$(-R,a)\not\in G(K)$, we conclude that
$$
\deg(G,K^{-},(\lambda^{-1},a))=\deg(G,K^{-},(-R,a))=0.
$$
From (\ref{equation28}), (\ref{equation32}),~(\ref{equation34})
and~(\ref{equation35}) we obtain a contradiction, and
Lemma~\ref{lm8} is proved.

\end{document}